\documentclass{siamltex}
\usepackage{appendix}
\usepackage{mathtext}
\usepackage{amsfonts,amsmath,amssymb,epsfig,graphicx,verbatim}
\usepackage{color}
\usepackage{appendix}
\usepackage{mathtools}
\usepackage{comment}
\usepackage[ruled]{algorithm2e}
\usepackage{algorithmic}

\DeclarePairedDelimiter{\ceil}{\lceil}{\rceil}

\newcommand{\evcomment}[1]{\textcolor{blue}{\textbf{#1}}} 

\title{Preconditioned eigensolvers for large-scale nonlinear Hermitian eigenproblems with variational characterizations. II. Interior eigenvalues\thanks{This version dated \today. This work was supported by the National Science Foundation under grants DMS-1115520, DMS-1418882 and DMS-1419100.}}

\author{Daniel B. Szyld\thanks{Department of Mathematics, Temple University (038-16), 1805 N. Broad Street, Philadelphia, PA 19122-6094, USA ({\tt szyld@temple.edu})}
\and 
Eugene Vecharynski\thanks{
Computational Research Division, Lawrence Berkeley National Laboratory,
One Cyclotron Road, MS 50F-1620L, Berkeley, CA 94720, USA ({\tt evecharynski@lbl.gov})
} 
\and 
Fei Xue\thanks{Department of Mathematics, University of Louisiana at Lafayette, P.O. Box 41010, Lafayette, LA 70504-1010, USA ({\tt fxue@louisiana.edu})}
}

\newtheorem{prop}[theorem]{Proposition}

\newtheorem{dftn}[theorem]{Definition}

\begin{document}

\maketitle




\setcounter{page}{1}

\begin{abstract}
We consider the solution of large-scale nonlinear algebraic Hermitian eigenproblems of the form $T(\lambda)v=0$ that admit a variational characterization of eigenvalues. These problems arise in a variety of applications and are generalizations of linear Hermitian eigenproblems $Av\!=\!\lambda Bv$. In this paper, we propose a Preconditioned Locally Minimal Residual (PLMR) method for efficiently computing interior eigenvalues of problems of this type. We discuss the development of search subspaces, preconditioning, and eigenpair extraction procedure based on the refined Rayleigh-Ritz projection. Extension to the block methods is presented, and a moving-window style soft deflation is described. Numerical experiments demonstrate that PLMR methods provide a rapid and robust convergence towards interior eigenvalues. The approach is also shown to be efficient and reliable for computing a large number of extreme eigenvalues, dramatically outperforming standard preconditioned conjugate gradient methods. 
\end{abstract}

{\scriptsize{\textbf{\!\!AMS subject classifications.} 65F15, 65F50, 15A18, 15A22.}}

\pagestyle{myheadings} \thispagestyle{plain}

\section{Introduction} \label{ch1_sect_intro}
Nonlinear Hermitian algebraic eigenproblems of the form $T(\lambda)v=0$ arise naturally in a variety of scientific and engineering applications. Many of these problems allow for a variational characterization (min-max principle) of some eigenvalues on certain intervals. Desirable properties of these eigenvalues and associated eigenvectors can be derived, and special methods can be developed to compute them efficiently. In Part I of this study \cite{Szyld.Xue.2014b}, we investigated Preconditioned Conjugate Gradient (PCG) methods for computing extreme eigenvalues of the nonlinear Hermitian eigenproblem that satisfy a variational principle. In this paper, to continue our study, we develop and explore Preconditioned Locally Minimal Residual (PLMR) methods for computing interior eigenvalues. Our exploration was motivated by a new class of preconditioned eigensolvers for linear eigenproblems \cite{Vecharynski.thesis.2011}, \cite{Vecharynski.Knyazev.2014}. 

Interior eigenvalues are intrinsically more difficult to compute than extreme eigenvalues. This is the case even for linear Hermitian problems $Av=\lambda Bv$. 
To devise an efficient and reliable interior eigenvalue solver, several issues need to be addressed. First, a good preconditioner $M$ approximating $A-\sigma B$ must be available, where $\sigma$ is a real shift close to the desired eigenvalues. Second, appropriate variants of 
subspace projection and eigenpair extraction should be used to provide a rapid and robust convergence towards interior eigenvalues. In particular, the extraction should identify and discard spurious Ritz values. 
Exactly the same types of challenges arise when solving nonlinear interior eigenproblems. 

For the first issue, we note that efficient and robust preconditioners for interior eigenvalue computations are typically constructed with indefinite matrices.  Their development is rather challenging in general, and is out of the scope of this paper. Nevertheless, several options are readily available, e.g., the incomplete $\mathrm{LDL^T}$ factorization \cite{Hagemann.Schenk.2006}, \cite{Schenk.etal.2008}, the absolute value preconditioning \cite{Vecharynski.Knyazev.2013}, or any iterative solver for a corresponding indefinite linear system. 
Here, we assume that a suitable preconditioner is at hand, and focus on the development of search subspaces and mechanisms for extracting approximate eigenpairs. 

The search subspaces suggested within the proposed PLMR method are given by certain preconditioned Krylov subspaces that are augmented with a search direction connecting the eigenvector approximations obtained in two consecutive iterations. 
Therefore, they can be viewed as a natural generalization of the PCG search subspaces~\cite{Szyld.Xue.2014b}.
The main difference with PCG is that the PLMR search subspace is based on an augmented Krylov subspace of a larger dimension, which enhances the robustness of convergence. 
Nevertheless, if a good preconditioner is available, the PLMR subspace does not have to be significantly larger than that of PCG. Therefore, similar to PCG, the eigenpair approximations in PLMR are normally extracted from subspaces of a small size. 

  

In order to extract interior eigenpairs of linear eigenproblems, the harmonic Rayleigh-Ritz pocedure \cite{Morgan.1991}, \cite{Sleijpen.Vorst.1996} is widely used. The same idea can be applied in the nonlinear case, which leads to the projected nonlinear eigenproblem $U^*T(\sigma)^*T(\nu)Uy=0$, where $U$ contains basis vectors of the search subspace. Then, the harmonic Ritz pairs $(\nu, U y)$ with $\nu$ close to $\sigma$ provide approximations to the desired interior eigenpairs of the original problem. 

The main disadvantage of the harmonic Rayleigh-Ritz projection is that it does not preserve symmetry, i.e., the projected problem $U^*T(\sigma)^*T(\nu)Uy=0$ is no longer Hermitian. The loss of symmetry is unlikely to be a major issue for linear problems. 
However, in the nonlinear case, solving projected eigenproblems that do not preserve the original structure could cause considerable complications. In particular, it may require special treatment of invariant pairs~\cite{Effenberg.2013}, and is likely to incur significant loss of accuracy in the final solutions. 

We avoid this issue by using the standard Rayleigh-Ritz procedure.
The resulting projected eigenproblem $U^*T(\nu)Uy=0$ is also Hermitian, with eigenvalues satisfying the variational principle. 
To remedy the slow convergence towards interior eigenvalues, which is commonly intrinsic to the standard 
Rayleigh-Ritz approach, we propose a simple strategy for discarding spurious Ritz values,
followed by an 
eigenvector refinement procedure 
(see \cite{Jia.1997} for linear eigenproblems) 
that stabilizes and accelerates the convergence.
Our experience shows that such a refined projection outperforms the harmonic Rayleigh-Ritz approach, 
and is crucial for maintaining robust convergence. We observe that the effects of the eigenvector refinement are significantly more pronounced in the nonlinear setting.


To understand the local convergence of the new method, we shall discuss a close connection between the PLMR search subspace 
and that of the basic Jacobi-Davidson (JD) method using the right-preconditioned GMRES as a solver for the correction equation. This connection allows derivation of the order of local convergence of PLMR, established under an assumption on the approximation property of the refined Rayleigh-Ritz procedure. Our analysis shows that PLMR with a search subspace of a fixed size converges linearly, and it exhibits a higher order of convergence if the search subspace is 
expanded with every new iteration. 

For the case where several eigenpairs are wanted, we present a block variant of the PLMR method, called BPLMR.
To the best of our knowledge, BPLMR is the first block variant of a preconditioned eigensolver for computing interior eigenvalues of nonlinear eigenproblems. 
In this algorithm, a special care is taken to ensure the robustness of the eigenvector refinement procedure, which is enhanced to avoid repeated convergence of semi-simple and clustered eigenvalues. Moreover, special attention is devoted to computing a large number of successive eigenvalues, for which a moving-window-style soft deflation strategy is described.
%


The proposed PLMR methods use several well-established techniques that contribute to fast and robust convergence towards interior eigenvalues. They share similarities with the nonlinear Arnoldi method \cite{Voss.2004} as both are ``preconditioned eigensolvers'' based on projections onto the Krylov-like search subspaces constructed by a preconditioned linear operator. They also possess features of the nonlinear Jacobi-Davidson method \cite{Betcke.Voss.2004}, \cite{Voss.2007} in the use of stabilized preconditioners. Consequently, we expect that PLMR performs at least as well as nonlinear Arnoldi and JD. In addition, the suggested eigenvector refinement procedure further improves the convergence, especially if the preconditioner is not very strong, or if clustered or semi-simple eigenvalues are desired. 

The paper is organized as follows. Section 2 reviews basics of nonlinear Hermitian eigenproblems, 
including a nonlinear variational principle. 
In Section~3, we propose a basic PLMR method for computing one interior eigenvalue of $T(\lambda)v=0$ around a given shift. 
Section~4 provides an insight into the connection between the search subspaces of PLMR and of the right-preconditioned \mbox{GMRES} used as an inner solver for a basic JD method, leading to a local convergence result for PLMR. In Section~5, we develop BPLMR for computing several interior eigenvalues simultaneously. Numerical results, which demonstrate the efficiency of PLMR methods, are presented in Section~6. 
Our conclusions can be found in Section~7.

\section{Nonlinear Hermitian eigenproblem and variational principle}\label{sect_preliminaries} In this section, we describe nonlinear algebraic Hermitian eigenproblems $T(\lambda)v=0$ that admit a variational characterization on an open interval $J$. Here, $T(\cdot): J \subset \mathbb{R}\rightarrow \mathbb{C}^{n \times n}$ maps a real scalar $\mu \in J$ continuously to the Hermitian matrix $T(\mu)$. The scalar $\lambda \in J$, for which $T(\lambda)$ is singular, is an eigenvalue of $T(\cdot)$ with a corresponding eigenvector $v \in \mathrm{null}\,T(\lambda) \setminus \{0\}$. 

To simplify our analysis, we assume, as in Part I \cite{Szyld.Xue.2014b}, that $T(\cdot)$ does not have infinite eigenvalues on $J$. This assumption is valid for most Hermitian eigenproblems encountered in practice. Under the assumption, $J=(a,b)$ containing all eigenvalues of interest is finite, where $a$ and $b$ are not eigenvalues of $T(\cdot)$. In certain circumstances, $T(\cdot)$ does have infinite eigenvalues (for instance, linear Hermitian eigenproblems $Av = \lambda Bv$ with a semi-definite $B$), but those eigenvalues generally have little physical relevance and thus are rarely desired. 

We start the description with several definitions.

\begin{dftn}
The \emph{Rayleigh functional} $\rho(\cdot): D \rightarrow J$ is a continuous mapping of a vector $x \in D\subset \mathbb{C}^n\setminus \{0\}$ to the {unique} solution $\rho(x) \in J$ of the equation $x^*T(\rho(x))x=0$. 
\end{dftn}

\begin{dftn}
Given $T(\cdot): J \subset \mathbb{R}\rightarrow \mathbb{C}^{n \times n}$, $J \subset \mathbb{R}$ is called an interval of \emph{positive or negative type}, if $(\mu-\rho(x))(x^*T(\mu)x)$ is constantly positive or constantly negative, respectively, for all $x \in D$ and all $\mu \in J$, $\mu \ne \rho(x)$. Both positive and negative type are \emph{definite type}. 
\end{dftn}

\begin{dftn}\label{ktheigenvalue}
A real scalar $\lambda$ is \emph{the $k$-th eigenvalue} of $T(\cdot)$ if zero is the $k$-th largest eigenvalue of the matrix $T(\lambda)$. Unless noted otherwise, the $k$-th eigenvalue is denoted as~$\lambda_k$. 
\end{dftn}

Necessary and sufficient conditions for $J$ to be of definite type are given as follows.

\begin{prop}[Proposition 2.4 in \cite{Szyld.Xue.2014b}]\label{intJdef}
Let $J=(a,b) \subset \mathbb{R}$ be finite, where $a,b$ are not eigenvalues of $T(\cdot)$, and let   $\rho: D\rightarrow J$ be the Rayleigh functional, where $D=\mathbb{C}^n\setminus\{0\}$. Then $J$ is an interval of positive \textup{(}negative\textup{)} type if and only if $T(a)$ is negative \textup{(}positive\textup{)} definite and $T(b)$ is positive \textup{(}negative\textup{)} definite. Assume that $T(\cdot)$ is continuously differentiable. Then $J$ is of positive \textup{(}negative\textup{)} type if $x^*T'(\rho(x))x > 0\:(<0)$ for all $x \in D$. If, in addition, $T(\cdot)$ is twice continuously differentiable and $x^*T''(\rho(x))x \ne 0$  for all $x \in D$, then $J$ is of positive \textup{(}negative\textup{)} type if and only if $x^*T'(\rho(x))x> 0\:(<0)$ for all $x \in D$. 
\end{prop}

On an interval of definite type, we have a variational characterization of eigenvalues of $T(\cdot)$ and the orthogonality of eigenvectors \cite{Hadeler.1968}\cite{Voss.2009}\cite{Voss.Werner.1982}. 

\begin{theorem}[Nonlinear Variational Principle]\label{nlvarprinciple}
Let $J \subset \mathbb{R}$ be finite and of definite type, and $T(\cdot)$ be continuously differentiable on $J$. Then there exist exactly $n$ eigenvalues $\{\lambda_k\}_{k=1}^n$ of $T(\cdot)$ on $J$ that satisfy a variational principle. Specifically, if $J$ is of positive type, then
\begin{eqnarray}\label{minmaxfm}
\lambda_k = \min\{\max\{\rho(x)\,|\,x \in S, x \ne 0\}\,|\,\mathrm{dim}(S)=k\}\quad \mbox{ and }\\ \nonumber
\lambda_k = \max\{\min\{\rho(x)\,|\,x \in S, x \ne 0\}\,|\,\mathrm{dim}(S)=n-k+1\}; \nonumber
\end{eqnarray} 
if $J$ is of negative type, then
\begin{eqnarray}\label{maxminfm}
\lambda_k = \max\{\min\{\rho(x)\,|\,x \in S, x \ne 0\}\,|\,\mathrm{dim}(S)=k\}\quad \mbox{ and }\\ \nonumber
\lambda_k = \min\{\max\{\rho(x)\,|\,x \in S, x \ne 0\}\,|\,\mathrm{dim}(S)=n-k+1\}. \nonumber
\end{eqnarray}
Moreover, there exist $n$ corresponding eigenvectors $\{v_k\}_{k=1}^n$ that form a basis of $\mathbb{C}^n$, and they are orthogonal with respect to the scalar-valued function $[\cdot,\cdot]$ defined as
\begin{eqnarray}\label{defdotprod}
[x,y] = \left\{\begin{array}{ll}y^*\big(T(\rho(y))-T(\rho(x))\big)x/\left(\rho(y)-\rho(x)\right)&\mbox{ if }\rho(x) \ne \rho(y)\\ y^*T'(\rho(x))x& \mbox{ if }\rho(x)=\rho(y)\end{array}\right..
\end{eqnarray} 
\end{theorem}

A natural corollary of the nonlinear variational principle \textup{(\ref{minmaxfm}) or (\ref{maxminfm})} is the nonlinear Cauchy interlacing theorem. 



\begin{theorem}[Nonlinear Cauchy interlacing theorem \cite{Szyld.Xue.2014b}]\label{cauchyilt}
Let $J=(a,b)$ be finite and of definite type, $T(\cdot)$ be continuously differentiable on $J$, and $U \in \mathbb{C}^{n \times m}$ contain $m$ linearly independent column vectors. Then the projected eigenproblem $U^*T(\nu)Uy = 0$ has exactly $m$ eigenpairs $\{(\nu_j,y_j)\}_{j=1}^m$ satisfying the nonlinear variational principle \textup{(\ref{minmaxfm}) or (\ref{maxminfm})}. In addition, if $J$ is of positive type, then $\lambda_j \leq \nu_j \leq \lambda_{n-m+j}$; if $J$ is of negative type, then $\lambda_{n-m+j} \leq \nu_j \leq \lambda_j$ \:$(1 \leq j \leq m)$.
\end{theorem}


From Definition \ref{ktheigenvalue} and the nonlinear variational principle \textup{(\ref{minmaxfm}) or (\ref{maxminfm})}, we see that the eigenvalues of the nonlinear Hermitian eigenproblem on an interval of definite type can be ordered in the same manner as for linear Hermitian eigenproblems. In particular, this ordering is needed when PLMR is used for computing many 
successive extreme eigenvalues, as described in Section 5.5. 


\section{The single-vector PLMR} In this section, we present a basic version of the PLMR method 
for computing an interior eigenvalue and the associated eigenvector of the nonlinear Hermitian eigenproblem.
We discuss the main building blocks of the method, including development of the search subspace, 
preconditioning, and extraction of the approximate eigenpair.

\subsection{Development of the search subspace}  Assume that $\lambda$ is a unique distinct
eigenvalue of $T(\cdot)$ closest to $\sigma \in J$, 
and let $v$ be a corresponding eigenvector. Suppose that in the $k$-th iteration we have an approximate eigenvector $x_k$. Our goal here is to develop a search subspace $\mathcal{U}_k$, from which a more accurate 
eigenvector approximation can be extracted. 

We start by reviewing the search subspace constructed within a variant of the PCG method, 
the Locally Optimal Preconditioned Conjugate Gradient (LOPCG) algorithm~\cite{Szyld.Xue.2014b}, 
for computing the lowest eigenvalue 
of $T(\cdot)$. 
%
Given the current eigenvector approximation $x_k$, LOPCG defines the search subspace as 
$$\mathcal{U}_k^{\mathrm{LOPCG}}=\mathrm{span}\{x_k, M^{-1}\nabla \rho(x_k), p_{k-1}\},$$
where 
$$\nabla \rho(x_k) = -\frac{2}{x^* T'\left(\rho(x_k)\right)x_k} T\left(\rho(x_k)\right)x_k$$ 
is the gradient of the Rayleigh functional $\rho(\cdot)$ at $x_k$ (see \cite[Proposition 3.1]{Szyld.Xue.2014b}), 
$M \approx T(\sigma)$ is a preconditioner, and $p_{k-1}$ is the  search direction connecting $x_{k-1}$ and $x_k$ 
($p_{-1} = 0$). 
Note that $T\left(\rho(x_k)\right)x_k$ defines the residual of the eigenproblem,
which is parallel to the gradient~$\nabla \rho(x_k)$. 

For the sake of simplicity, we let $\rho_k=\rho(x_k)$ when there is no danger of confusion. The above search subspace can then be written as 
\begin{equation}\label{eq:lopcg}
\mathcal{U}_k^{\mathrm{LOPCG}}=\mathcal{K}_2\left(M^{-1}T(\rho_k),x_k\right)+\mathrm{span}\{p_{k-1}\}, 
\end{equation}
where $\mathcal{K}_m(A,b) = \mathrm{span}\left\{b, Ab, \ldots, A^{m-1}b\right\}$
denotes an $m$-dimensional Krylov subspace~\cite{Saad.book.2003}. This three-dimensional search subspace, unfortunately, is not effective for computing interior eigevalues. 
The main issue is that the convergence towards these eigenvalues can be fairly slow, 
and thus a search subspace of a larger dimension is needed to stabilize and accelerate the convergence. This is especially evident if the preconditioner 
$M$ is not very strong. 

In order to properly enlarge the LOPCG subspace~\eqref{eq:lopcg}, 
we consider the subspace 
\begin{eqnarray}\label{spacePLMR}
\mathcal{U}_k=\mathcal{K}_{m+1}(M^{-1}T(\rho_k),x_k)+\mathrm{span}\{p_{k-1}\},
\end{eqnarray}
where $\mathcal{K}_{m+1}(M^{-1}T(\rho_k),x_k)$ generalizes the limit space 
$\mathcal{K}_{m+1}\left(M^{-1}(A-\rho_k B),x_k\right)$ of the Generalized Davidson method that restarts every $m$ steps \cite{Ovtchinnikov.2003}.
The augmentation of this Krylov subspace with the search direction $p_{k-1}$ is expected to accelerate the convergence 
as it does for PCG methods. 

\subsection{Stabilization of preconditioning} 
A drawback of the search subspace~(\ref{spacePLMR}) is that it can potentially suffer from numerical instabilities. 
To see this, let
$M = T(\sigma)$ and $\rho_k = \sigma$. Then $M^{-1}T(\rho_k) = I$, and the search subspace 
degenerates 
to $\mathrm{span}\left\{x_k,p_{k-1}\right\}$. Therefore, an algorithm based on the search subspace~(\ref{spacePLMR}) stagnates, i.e., cannot generate any improvement in eigenvector approximation. In practice, stagnation could arise whenever $M$ is a good approximation to $T(\sigma)$ and $\rho_k$ is sufficiently close to $\sigma$. The same issue is known for the Davidson type methods for linear eigenproblems, which has been fixed by the Jacobi-Davidson (JD) algorithm \cite{Fokkema.etal.1998}, \cite{Sleijpen.Vorst.1996}. 


A key ingredient contributing to the robustness of the JD methods is the modification of the preconditioning procedure in such a way that it is performed through solution of a correction equation rather than a direct application of $M^{-1}$. In particular, for nonlinear eigenproblems, 
the correction equation  of the basic JD method\footnote{The basic variant of JD forms the new approximation as $x_{k+1}= x_k +\Delta x_k$, where $\Delta x_k$ is an approximate solution to the correction equation; no subspace expansion and projection is involved. It is also referred to as single-vector JD or simplified JD in literature; see, e.g., \cite{Freitag.Spence.2008}, \cite{Hochstenbach.Notay.2009}.} is of the form 
\begin{eqnarray}\label{JDcrteqn}
\Pi_1 T(\rho_k) \Pi_2 \Delta x_k = -T(\rho_k)x_k,
\end{eqnarray}
where $\Pi_1$ and $\Pi_2$ are properly chosen projectors, such that $T(\rho_k)x_k \in \mathrm{range}(\Pi_1)$ and $\Pi_2 \Delta x_k = \Delta x_k$; see, e.g., \cite{Betcke.Voss.2004}, \cite[Chapter 6.2]{Schreiber.thesis}, \cite{Szyld.Xue.2013}. For Hermitian $T(\cdot)$, one can choose 
\begin{eqnarray}\label{JDprj}
\Pi_1 = I-\frac{T'(\rho_k)x_kx_k^*}{x_k^*T'(\rho_k)x_k}\quad\mbox{ and }\quad \Pi_2 = \Pi_1^*=I-\frac{x_kx_k^*T'(\rho_k)}{x_k^*T'(\rho_k)x_k}.
\end{eqnarray} 
In this case, the coefficient matrix in (\ref{JDcrteqn}) is Hermitian, and thus efficient preconditioned linear solvers such as MINRES~\cite{Paige.Saunders:75} or SQMR \cite{Freund.Nachtigal.1995} can be applied. 
It can be shown that the exact solution $\Delta x_k$ of (\ref{JDcrteqn}) with the projectors defined in (\ref{JDprj}) satisfies $$x_k+\Delta x_k = \frac{x_k^*T'(\rho_k)x_k}{x_k^*T'(\rho_k)T(\rho_k)^{-1}T'(\rho_k)x_k}T(\rho_k)^{-1}T'(\rho_k)x_k,$$ which is parallel to the new iterate $x_{k+1}=T(\rho_k)^{-1}T'(\rho_k)x_k$ obtained from the Rayleigh functional iteration; see \cite[Chapter 4.3]{Schreiber.thesis}, \cite{Szyld.Xue.2013}.

Motivated by the structure of the coefficient matrix in the correction equation~(\ref{JDcrteqn}), we modify 
the preconditioner $M \approx T(\sigma)$ by multiplying it with the projectors in~(\ref{JDprj}).
This replaces $M$ with a stabilized preconditioner
\begin{eqnarray}\label{pcdMp}
M_{\Pi} = \Pi M \Pi^*,
\end{eqnarray}
where $\Pi = \Pi_1$, as defined in (\ref{JDprj}). Thus, the preconditioner is now applied to a vector through $M_{\Pi}^{\dagger} \equiv \Pi M^{-1} \Pi^*$ rather than~$M^{-1}$. 


The precise formula that describes the action of~$M_{\Pi}^{\dagger}$ on a given vector can be derived as follows. We consider vectors $y$ and $b$ such that $b = M_{\Pi}y=\Pi M \Pi^* y$, where $b \in \mathrm{range}(\Pi)$, and $y \perp T'(\rho_k)x_k$, i.e., $\Pi^*y = y$. These assumptions are standard in the preconditioning for JD, identical to those used in \cite{Szyld.Xue.2013}, \cite{Voss.2014}. Equivalently, 
\begin{eqnarray}\nonumber
b = \left(I-\frac{T'(\rho_k)x_kx_k^*}{x_k^*T'(\rho_k)x_k}\right)M y = My-T'(\rho_k)x_k \frac{x_k^*My}{x_k^*T'(\rho_k)x_k},
\end{eqnarray}
and it follows that
\begin{eqnarray}\label{prcd_drv1}
y = M^{-1}b + M^{-1}T'(\rho_k)x_k \frac{x_k^*My}{x_k^*T'(\rho_k)x_k}. 
\end{eqnarray}
The orthogonality $y \perp T'(\rho_k)x_k$ implies that
\begin{eqnarray}\nonumber
0 = x_k^*T'(\rho_k)y = x_k^*T'(\rho_k)M^{-1}b+x_k^*T'(\rho_k)M^{-1}T'(\rho_k)x_k\frac{x_k^*My}{x_k^*T'(\rho_k)x_k},
\end{eqnarray}
from which we have
\begin{eqnarray}\label{prcd_drv2}
\frac{x_k^*My}{x_k^*T'(\rho_k)x_k}=-\frac{x_k^*T'(\rho_k)M^{-1}b}{x_k^*T'(\rho_k)M^{-1}T'(\rho_k)x_k}.
\end{eqnarray}
Thus, after substituting \eqref{prcd_drv2} into \eqref{prcd_drv1}, we obtain
\begin{eqnarray}\nonumber
y = M_{\Pi}^{\dagger}b &=& M^{-1}b-\frac{x_k^*T'(\rho_k)M^{-1}b}{x_k^*T'(\rho_k)M^{-1}T'(\rho_k)x_k}M^{-1}T'(\rho_k)x_k \\ \label{pcdmvp0}
&=&\left(I-\frac{M^{-1}T'(\rho_k)x_k\,x_k^*T'(\rho_k)}{x_k^*T'(\rho_k)M^{-1}T'(\rho_k)x_k}\right)M^{-1}b.
\end{eqnarray}

In contrast to $M^{-1}$, the operator $M_{\Pi}^{\dagger}$ does not cancel out with 
the matrix $T(\rho_k)$ and, instead, applies~$M^{-1}$ to~$T'(\rho_k)x_k$, 
magnifying the desired eigenvector component. 
Note that $M^{-1}T'(\rho_k)x_k$ in~\eqref{pcdmvp0} can be computed only once and further used to evaluate vectors of the form $M_{\Pi}^{\dagger}T(\rho_k)z$, for any $T(\rho_k)z \in \mathrm{range}(\Pi)$. This observation is needed for the construction of the PLMR search subspace. 

%

Finally, given a stabilized preconditioner~\eqref{pcdMp}, whose action on a vector is expressed in~\eqref{pcdmvp0},
we define the PLMR search subspace~as
\begin{eqnarray}\label{spacePLMR2}
\mathcal{U}_k^{\mbox{\scriptsize{PLMR($m$)}}}=\mathcal{K}_{m+1}(M_{\Pi}^{\dagger}T(\rho_k),x_k)+\mathrm{span}\{p_{k-1}\}, 
\end{eqnarray}
which is exactly (\ref{spacePLMR}) with $M$ replaced by $M_{\Pi}$.
Our numerical experience confirms that the PLMR version built upon (\ref{spacePLMR2}) 
indeed
tends to be significantly more robust than that based on (\ref{spacePLMR}).
Therefore, throughout, we only use (\ref{spacePLMR2}), constructed with the stabilized preconditioner, as the search subspace for the PLMR algorithm.  

\subsection{Subspace projection and extraction} Given the PLMR search subspace (\ref{spacePLMR2}), we now consider the projection of the original eigenproblem onto this subspace, and describe the extraction of a new eigenvector approximation. 

Similar to the linear setting, the standard Rayleigh-Ritz procedure is ideal in preserving symmetry and is most suitable for computing extreme eigenvalues, but it generally exhibits very slow convergence towards interior eigenvalues. As a remedy, the harmonic Rayleigh-Ritz scheme could be used. However, the main 
disadvantage of this approach is that it does not preserve symmetry of the original eigenproblem. In general, algorithms for solving interior eigenvalues of nonlinear eigenproblems without symmetry are significantly more complicated and tend to be less robust than those for solving nonlinear Hermitian eigenproblems admitting a variational principle; see, e.g.,~\cite{Voss.2010}\cite{Voss.2014} and references therein. In addition, solving nonlinear projected eigenproblems that fail to preserve the structure (non-Hermitian in our case) 
can lead to a significant loss of accuracy in the final eigenpair approximations.

To resolve this issue, we propose using the standard Rayleigh-Ritz projection, followed by a procedure to detect and discard spurious Ritz values, and a refinement step 
to further improve the quality of eigenvector approximation. 
Note that spurious Ritz values are Ritz values close to the desired shift $\sigma$, but they correspond to poor eigenvector approximations, typically given by linear combinations of eigenvectors associated with eigenvalues outside the interval of interest. They often arise frequently when interior eigenvalues are sought.

Specifically, let $U_k \in \mathbb{C}^{n \times (m+2)}$ contain orthonormal basis vectors of the search subspace~\eqref{spacePLMR2}.
The Rayleigh-Ritz scheme then leads to the projected Hermitian eigenproblem 
\begin{equation}\label{eq:rr}
U_k^*T(\nu)U_k y = 0,
\end{equation}
which also admits a variational characterization of its eigenvalues (i.e., of the Ritz values) 
satisfying the nonlinear Cauchy interlacing theorem (Theorem \ref{cauchyilt}). The projected eigenproblem can be solved, e.g., by PCG methods based on the variational principle \cite{Szyld.Xue.2014b}.

After forming the projected problem~\eqref{eq:rr}, the following approach is used to obtain an approximate eigenpair. 
First, we solve~\eqref{eq:rr} for the $r$ successive Ritz values $\nu_{1},\ldots, \nu_{r}$ closest to $\sigma$ and the associated eigenvectors $y_1,\ldots,y_r$, and compute the corresponding Ritz vectors $z_i = U_k y_i$ ($1\leq i \leq r$). We then order $\{\nu_i\}$ according to the residual norms of the respective Ritz pairs, such that for any $i$, $j$ with $1 \leq i<j \leq r$,
\begin{equation}\label{eq:sort}
\frac{\big\|T\left(\nu_i\right)z_i\big\|_2}{\big\|T\left(\nu_i\right)\big\|_F\big\|z_i\big\|_2} \leq \frac{\big\|T\left(\nu_j\right)z_j\big\|_2}{\big\|T\left(\nu_j\right)\big\|_F\big\|z_j\big\|_2}.
\end{equation}
Next, we take $s$ Ritz pairs $(\nu_i, z_i)$ of minimal eigenresidual norm from the $r$ candidates, 
and choose the Ritz value, say $\nu_{\ell}$ ($1 \leq \ell \leq s$), that is closest to $\sigma$. 

The motivation of the above approach is well founded. We first find a relatively large number, $r$, of Ritz values near $\sigma$, so that a good eigenvalue approximation is included in this set. Other Ritz values, not selected, are relatively far from $\sigma$, and cannot represent accurate approximations to the desired eigenvalue. The ordering of Ritz pairs in terms of eigenresidual norm tends to put promising Ritz pairs to the front of the ordered set 
and others to the end. This step aims to filter out spurious Ritz pairs, i.e., those with Ritz values close to $\sigma$ but with 
large eigenresidual norms.
Such pairs commonly arise in the Rayleigh-Ritz projection for computing interior eigenvalues, and are
excluded from further consideration due to the proposed ordering. 
As a result, we have a fairly good chance that a promising Ritz pair is included in the set of $s$ candidates with minimal eigenresidual norm. Finally, the Ritz value $\nu_\ell$ closest to $\sigma$ is selected from the $s$ candidates.
In our implementation, by default, $r=\min(m+1,\max(5,\ceil{(m+1)/2}))$, and $s = 2$.

\begin{algorithm}[htbp]
\begin{small}
\begin{center}
  \begin{minipage}{5in}
\begin{tabular}{p{0.5in}p{4.1in}}
{\bf Input}:  &  \begin{minipage}[t]{4.0in}
An initial eigenvector approximation $x_0 \in \mathbb{C}^{n}\setminus\{0\}$, a preconditioner $M$, a shift $\sigma \in \mathbb{R}$, and integers $r,\,s > 0$;
                  \end{minipage} \\
{\bf Output}:  &  \begin{minipage}[t]{4.0in}
                 An eigenpair $(\lambda,v)$, where $\lambda$ is the eigenvalue of $T(\cdot)$ closest to $\sigma$;
                  \end{minipage}
\end{tabular}
\begin{algorithmic}[1]
\STATE Set $k \gets 0$ and $p_{-1} \gets [\,]$. Compute $\rho_0 = \rho(x_0)$. 
\WHILE {convergence not reached}
  \STATE If $k > 0$, then $p_{k-1} \gets x_k-x_{k-1}$. 
  \STATE Compute an orthonormal basis $U_k$ of the search subspace~\eqref{spacePLMR2}, where the action of the    
   preconditioner $M_{\Pi}$ is given by~\eqref{pcdmvp0}.
  \STATE Solve the Rayleigh-Ritz projected eigenproblem~\eqref{eq:rr}. 
  \STATE Select the $r$ Ritz values closest to $\sigma$, and sort the corresponding Ritz pairs according to  
         their eigenresidual norms~\eqref{eq:sort}. Then choose the $s$ Ritz pairs with minimal eigenresidual,  
          and use them to identify the Ritz value $\nu_{\ell}$ closest to $\sigma$. 
  \STATE  Compute the right singular vector $y_{MR}$ associated with the smallest
          singular value of the matrix $T(\nu_{\ell})U_k$. 

   \STATE Set $x_{k+1} \gets U_k y_{MR}$; $\rho_{k+1} \gets \rho(x_{k+1})$. 
   \STATE   Normalize $x_{k+1}$, such that $x_{k+1}^*T'(\rho_{k+1})x_{k+1}=1$.
   \STATE $k \gets k + 1$. Check convergence of $(\rho_k,x_k)$. 
\ENDWHILE
\STATE Set $\lambda \gets \rho_{k}$; $v \gets x_{k}$. Return $(\lambda, v)$.
\end{algorithmic}
\end{minipage}
\end{center}
\end{small}
  \caption{The PLMR($m$) algorithm for a Hermitian eigenproblem $T(\lambda)v = 0$}
  \label{alg:plmr}
\end{algorithm}

By construction, the selected interior Ritz pair $(\nu_\ell,z_\ell)$ 
has a reasonably small eigenresidual. However, in most cases, 
it can be further significantly improved. 
To this end, we refine the Ritz vector by substituting it with 
a new eigenvector approximation that delivers a minimal residual.  
That is, we solve
\begin{equation}\label{eq:ref}
{y}_{MR} = \mathrm{argmin}_{\|y\|=1}\left\|T(\nu_{\ell})U_k y\right\|_2,
\end{equation}
and set the new iterate to $x_{k+1}= U_k {y}_{MR}$. The corresponding eigenvalue
approximation is then given by the Rayleigh functional $\rho_{k+1}$ evaluated at $x_{k+1}$.

Problem~\eqref{eq:ref} can be approached by finding the smallest singular value of the matrix $T(\nu_{\ell})U_k \in \mathbb{C}^{n \times (m+2)}$ and its right singular vector, or equivalently by solving the linear eigenproblem $U_k^*T(\nu_{\ell})^*T(\nu_{\ell})U_ky = \eta y$ for the eigenvector corresponding to the smallest eigenvalue. 
Note that the entire strategy described above is a direct generalization of the 
refinement procedure of~\cite{Jia.1997} to the case of nonlinear eigenproblems.
As we shall see in Section~\ref{sect_num_exp}, step~\eqref{eq:ref} indeed turns out to be crucial for stabilizing the convergence of PLMR.
The whole PLMR scheme is summarized in Algorithm~1.

\section{Local convergence analysis}\label{sect_ca}
In this section, we provide an analysis of the PLMR algorithm, explaining the conditions that guarantee its convergence and how rapidly it may converge. Our analysis is based on a close connection between the search subspaces developed by PLMR and the basic JD method, and on certain assumptions about the performance of the subspace projection and extraction. 

Specifically, let $(\rho_k,x_k)$ be the current approximation to the desired eigenpair $(\lambda,v)$, where $\lambda$ is the eigenvalue of $T(\cdot)$ closest to $\sigma$. Recall from (\ref{JDcrteqn}) the basic JD correction equation $$\Pi T(\rho_k)\Pi^* \Delta x_k = -T(\rho_k) x_k,$$ and assume that the preconditioner (\ref{pcdMp}) is used for a Krylov subspace method with right-preconditioning to solve this equation. In iteration $m$, the Krylov subspace developed for the preconditioned linear system is thus $$\mathcal{K}_{m}\left(\Pi T(\rho_k) \Pi^* M_{\Pi}^{\dagger}, T(\rho_k)x_k\right),$$ where $\Pi = I-\frac{T'(\rho_k)x_kx_k^*}{x_k^*T'(\rho_k)x_k}$. Due to the right-preconditioning, the approximate solution $\Delta x_k$ of the original JD correction equation \eqref{JDcrteqn} lies in $M_{\Pi}^{\dagger}\mathcal{K}_{m}\left(\Pi T(\rho_k)\Pi^* M_{\Pi}^{\dagger},T(\rho_k)x_k\right)$, and therefore the new approximation $x_{k+1}=x_k +\Delta x_k$ lies in 
\begin{eqnarray}\label{Kspace}
&& \mathrm{span}\{x_k\} +M_{\Pi}^{\dagger}\mathcal{K}_{m}\left(\Pi T(\rho_k)\Pi^* M_{\Pi}^{\dagger},T(\rho_k)x_k\right) \\ \nonumber
&=& \mathrm{span} \left\{x_k,M_{\Pi}^{\dagger}T(\rho_k)x_k, M_{\Pi}^{\dagger}\Pi T(\rho_k)\Pi^* M_{\Pi}^{\dagger}T(\rho_k)x_k,\ldots, M_{\Pi}^{\dagger}\left(\Pi T(\rho_k)\Pi^* M_{\Pi}^{\dagger}\right)^{m-1}T(\rho_k)x_k\right\}\\ \nonumber
&=&\mathrm{span} \left\{x_k,M_{\Pi}^{\dagger}T(\rho_k)x_k,\left(M_{\Pi}^{\dagger} T(\rho_k) \right)^{2}x_k\ldots, \left(M_{\Pi}^{\dagger} T(\rho_k) \right)^{m}x_k\right\}= \mathcal{K}_{m+1}\left(M_{\Pi}^{\dagger} T(\rho_k), x_k\right), \nonumber
\end{eqnarray}
where we used the identity $$M_{\Pi}^{\dagger}\Pi=\Pi^*M_{\Pi}^{\dagger}=M_{\Pi}^{\dagger},$$ that can be derived from (\ref{JDprj}) and (\ref{pcdmvp0}) without much difficulty. Clearly, the subspace \eqref{Kspace} is a proper subspace of the PLMR($m$) search subspace \eqref{spacePLMR2}, an augmented version of \eqref{Kspace}. This observation is summarized in the following lemma.

\begin{lemma}\label{lemmaJDPLMR}
Given the same current iterate $x_k$, basic \textup{JD} with correction equation \eqref{JDcrteqn} delivers a new eigenvector approximation $x^{\mathrm{JD}}_{k+1}$ lying in the search subspace where \textup{PLMR($m$)} extracts its new iterate $x^{\mathrm{PLMR}}_{k+1}$.
\end{lemma}

Consequently, if $x^{\mathrm{PLMR}}_{k+1}$ is of the same quality as $x^{\mathrm{JD}}_{k+1}$, then the convergence of PLMR can be established as a corollary of the local convergence theorem of basic JD, already shown in our problem setting \cite[Theorems 7, 11]{Szyld.Xue.2013}. Whether the new iterates of the two methods are comparable in quality depends on the approximation properties of the refined Rayleigh-Ritz projection used in PLMR, which have been established for standard linear eigenproblems; see, e.g., \cite[Chapter 4.4]{Stewart.book.2001} and references therein. 

Let $v$ be the desired eigenvector, and $U$ contain basis vectors for the eigensolver search subspace. Roughly speaking, under certain typically non-stringent conditions, $\angle(v,Uy)$, the angle between $v$ and the corresponding Ritz or refined vector $Uy$, is proportional to $\angle(v,\mathrm{range}(U))$. For nonlinear eigenproblems $T(\lambda)v=0$, a complete study of similar properties of these techniques is beyond the scope of this paper. Nevertheless, in our numerical experiments, we find that $\angle(v,Uy)$ is also proportional to $\angle(v,\mathrm{range}(U))$  consistently. Therefore, we assume that this property holds, and give a major local convergence result of PLMR.
\begin{theorem}
Let $(\lambda, v)$ be a simple eigenpair of the nonlinear Hermitian eigenproblem $T(\lambda)v=0$, where $v$ is normalized such that $v^*T'(\lambda)v=1$. Assume that there exist a $\delta > 0$ and a corresponding $\xi > 0$, such that for any eigenpair approximation $(\mu,x)$ sufficiently close to $(\lambda,v)$, namely, with $\left\|\left[\!\begin{array}{c}x \\ \mu\end{array}\!\right]\!-\!\left[\!\begin{array}{c}v \\ \lambda\end{array}\!\right]\right\| \leq \delta$, we have $\|T'(\mu)x\| \leq \xi$. Let $x_k = \gamma_k(c_k v+s_k g_k)$ be the eigenvector approximation obtained in the $k$-th iteration of \textup{PLMR}, where $\gamma_k$, $c_k$ and $s_k$ are the generalized norm of $x_k$, generalized cosine and sine of $\angle(x_k,v)$, respectively \textup{(}see Appendix\textup{)}. Suppose that $\angle(x_0,v)$ is sufficiently small, such that $\left\|\left[\!\begin{array}{c}x_0 \\ \rho(x_0)\end{array}\!\right]\!-\!\left[\!\begin{array}{c}v \\ \lambda\end{array}\!\right]\right\| \leq \delta$. For each $x_k$, assume that the refined  projection extracts a new eigenvector approximation $x_{k+1}$, such that $\sin \angle(v,x_{k+1}) \leq C \sin \angle (v, \mathcal{U}^{\mathrm{PLMR}}_k)$ for a small constant $C$ independent of $k$. Assume that the \textup{JD} correction equation \textup{(\ref{JDcrteqn})} is solved by right-preconditioned \textup{GMRES($m_k$)} with the preconditioner defined in \textup{(\ref{pcdMp})}. Let $\tau^{(\alpha)}_k = \tau^{(\alpha)}_0$ be a sufficiently small and fixed tolerance, and $\tau^{(\beta)}_k \leq C_{\beta} \frac{|s_k|}{|c_k|}$ and $\tau^{(\gamma)}_k \leq C_{\gamma} \frac{|s_k|^2}{|c_k|^2}$ be decreasing sequences of tolerances, where $C_\beta$ and $C_\gamma$ are sufficiently small constants independent of $k$. For each $k$, assume that $m_k$ is sufficiently large, such that one cycle of \textup{GMRES($m_k$)} delivers an approximate solution of \textup{(\ref{JDcrteqn})} satisfying the relative tolerance $\tau^{(\alpha)}_k$, $\tau^{(\beta)}_k$ or $\tau^{(\gamma)}_k$, respectively. Then \textup{PLMR($m_k$)} converges towards $(\lambda,v)$ at least linearly, quadratically or cubically, respectively. 
\end{theorem}
\begin{proof} Given the above assumptions, it is shown in Theorems 7 and 11 in \cite{Szyld.Xue.2013} that the basic JD method with approximate inner linear solves that satisfy the tolerances $\tau_k^{(\alpha)}$, $\tau_k^{(\beta)}$ and $\tau_k^{(\gamma)}$, respectively, converges locally towards $(\lambda, v)$ linearly, quadratically and cubically, respectively. Note that basic JD generates the new approximation $x^{\textup{JD}}_{k+1} \in \mathcal{U}^{\mathrm{PLMR}}_k$ as shown in (\ref{Kspace}), and therefore $\sin \angle(v, x^{\textup{JD}}_{k+1}) \geq \sin \angle (v, \mathcal{U}^{\mathrm{PLMR}}_k)$. By assumption, the refined projection of PLMR delivers the new eigenvector approximation $x^{\textup{PLMR}}_{k+1}$ satisfying $\sin \angle(v,x^{\textup{PLMR}}_{k+1}) \leq C \sin \angle (v, \mathcal{U}^{\mathrm{PLMR}}_k)$, and therefore $\sin \angle(v,x^{\textup{PLMR}}_{k+1}) \leq C \sin \angle(v, x^{\textup{JD}}_{k+1})$. The convergence of PLMR thus follows directly from that of basic JD. 
\end{proof}

\section{Block PLMR} In this section, we consider simultaneous computation of a few interior eigenvalues and their eigenvectors.
To this end, we develop a block variant of PLMR, referred to as BPLMR. We shall see that most of the techniques used in PLMR can be extended directly to the block case. We also discuss deflation techniques and describe their application to computing large numbers 
of successive eigenpairs.    

\subsection{Search subspace and preconditioning} Assume that we want to find the $q$ eigenvalues  closest to $\sigma$, namely, $\lambda_{1}, \ldots, \lambda_{q}$, such that $|\lambda_1-\sigma| \leq \ldots, \leq |\lambda_{q}-\sigma|$, together with the associated eigenvectors $v_{1},\ldots,v_{q}$\footnote{To facilitate the description of interior eigenvalue computation, the numbering of eigenvalues here is different from the natural order defined in Definition 2.3.}. 
Let $X_k =[x_k^{(1)}\, \ldots \: x_k^{(q)}] \in \mathbb{C}^{n \times q}$ 
be the block of eigenvector approximations at 
iteration $k$, and let $\Phi_k = \mathrm{diag}(\rho_k^{(1)}, \ldots\: \rho_k^{(q)})$ denote a diagonal matrix of the Rayleigh functional values $\rho_k^{(i)}=\rho(x_k^{(i)})$. 
We define the block of eigenresiduals
$$\mathbb{T}(X_k,\Phi_k)=[T(\rho_k^{(1)})x_k^{(1)}\,\ldots\:T(\rho_k^{(q)})x_k^{(q)}],$$ 
and can hence construct a 
LOBPCG-type
search subspace spanned by the columns of $X_k$, $\mathbb{T}(X_k,\Phi_k)$, and $P_{k-1}$, where
$P_{k-1}$ carries information about the approximate eigenvectors in the previous step ($P_{-1} = \mathbf{0}$); see Part I of this study \cite{Szyld.Xue.2014b}. 
Similar to the single-vector case, such a LOBPCG subspace can be further expanded to better accommodate approximations 
of the interior eigenpairs, leading to the BPLMR search subspace  
\begin{eqnarray}\label{spaceBPLMR}
\mathcal{U}^{\mbox{\scriptsize{BPLMR($m$)}}}_k=\mathcal{K}_{m+1}\left(\mathbb{M}^{\dagger}_{\mathbf{\Pi}}\mathbb{T}(\,\cdot\,,\Phi_k), X_k\right) + \mathrm{range}\left(P_{k-1}\right),
\end{eqnarray}
where $\mathcal{K}_{m+1}\left(\mathbb{M}^{\dagger}_{\mathbf{\Pi}}\mathbb{T}(\,\cdot\,,\Phi_k), X_k\right)$ is the 
block Krylov subspace generated by the starting block $X_k$ and the linear operator 
$\mathbb{L}_k(\cdot) \equiv \mathbb{M}^{\dagger}_{\mathbf{\Pi}}\mathbb{T}(\,\cdot\,,\Phi_k)$. That is, 
\[
\mathcal{K}_{m+1}\left(\mathbb{M}^{\dagger}_{\mathbf{\Pi}}\mathbb{T}(\,\cdot\,,\Phi_k), X_k\right) = \mathrm{range}\big\{X_k,\mathbb{L}_k(X_k), \mathbb{L}_k\left(\mathbb{L}_k(X_k)\right), \ldots, \mathbb{L}_k^m(X_k)\big\},
\]
where $\mathbb{L}^m(\cdot)$ stands for the composition of $\mathbb{L}(\cdot)$ with itself for $m$ times, and $\mbox{range}(X)$ denotes the column space of $X$. 

By analogy with~\eqref{JDprj} and \eqref{pcdMp}, in~\eqref{spaceBPLMR}, we introduce a stabilized preconditioner 
\begin{eqnarray}\label{stabpcd1}
\mathbb{M}_{\mathbf{\Pi}} = \mathbf{\Pi}M\mathbf{\Pi}^*, 
\end{eqnarray}
where 
\[
\mathbf{\Pi} = I - Z_k(X_k^*Z_k)^{-1}X_k^*, \quad Z_k \equiv \mathbb{T}'(X_k,\Phi_k) =[T'(\rho_k^{(1)})x_k^{(1)}\,\ldots\:T'(\rho_k^{(q)})x_k^{(q)}].
\]
The above projector $\mathbf{\Pi}$ is a direct extension of the one defined in~\eqref{JDprj} to the block case.
Similar to~\eqref{pcdmvp0}, the action of $\mathbb{M}_{\mathbf{\Pi}}^{\dagger}$ on a block of vectors $B \in \mathrm{range}\left(\mathbb{M}_{\mathbf{\Pi}}\right)$
can be expressed as 
\begin{eqnarray}\label{pcdmvp1}
\mathbb{M}_{\mathbf{\Pi}}^{\dagger}B = \left(I-M^{-1}Z_k(Z_k^*M^{-1}Z_k)^{-1}Z_k^*\right)M^{-1}B.
\end{eqnarray}
Clearly,~\eqref{spaceBPLMR} represents a sum of $q$ PLMR search subspaces~\eqref{spacePLMR2} with starting vectors 
$x_k^{(i)}$ ($1 \leq i \leq q$) and the preconditioner $\mathbb{M}_{\mathbf{\Pi}}$ in~(\ref{stabpcd1}), 
and is of dimension $(m+2)q$ in general.

The block search direction $P_{k-1}$ can have several possible formulations. 
One option is to define $P_{k-1} = X_k-X_{k-1}$, which represents 
a direct generalization of the single-vector directions $p_{k-1} = x_k - x_{k-1}$ used is PLMR.
An alternative formulation can be given by 
\begin{equation}\label{eq:P}
P_{k-1}=X_k-X_{k-1}\left(X_{k-1}^*X_{k-1}\right)^{-1}X_{k-1}^*X_k,
\end{equation} 
which is a residual of the least squares problem $\min_{G \in \mathbb{C}^{q \times q}} \|X_k-X_{k-1}G\|_F$. 
Hence, definition~\eqref{eq:P} guarntees that $P_{k-1}$ has the smallest norm columnwise for all blocks of the form $X_k - X_{k-1}G$, 
where $G \in \mathbb{C}^{q \times q}$. 

In exact arithmetic, the two variants of $P_{k-1}$ lead to  
the same search subspace, because
\begin{eqnarray}\nonumber
\mathcal{U}^{\mbox{\scriptsize{BPLMR($m$)}}}_k&=&\mathrm{range}\{P_{k-1}\} + \mathcal{K}_{m+1}\left(\mathbb{M}^{\dagger}_{\mathbf{\Pi}}\mathbb{T}(\,\cdot\,,\Phi_k), X_k\right) \\ \nonumber
&=& \mathrm{range}\{X_{k-1}\}+\mathcal{K}_{m+1}\left(\mathbb{M}^{\dagger}_{\mathbf{\Pi}}\mathbb{T}(\,\cdot\,,\Phi_k), X_k\right).
\end{eqnarray}
In practice, BPLMR($m$) working with either version of $P_{k-1}$ converges equally rapidly in most cases, but formulation~\eqref{eq:P} tends to perform slightly better occasionally. We have no complete understanding of this, but have an intuitive explanation. The individual eigenvector approximations in $X_k$ are usually properly ordered, e.g., by the distances between the corresponding eigenvalue approximations and $\sigma$. As the algorithm proceeds, the ordering of some eigenvector approximations could change due to the change of their eigenvalue approximations. When such a change occurs, $P_{k-1}=X_k\!-\!X_{k-1}$ generates poor search directions that represent the difference between approximations to distinct eigenvectors in two consecutive iterations. By contrast, the least squares problem finds a matrix $G = (X_{k-1}^*X_{k-1})^{-1}X_{k-1}^*X_k$ that `reorders' the columns of $X_k$ so that $X_kG$ aligns columnwise with $X_k$, and thus $P_{k-1} = X_k\!-\!X_{k-1}G$ represents the the difference between the \emph{subspaces} spanned by the two block iterates, and it is more likely to be numerically favorable.



\subsection{Subspace projection and extraction}\label{sec_BPLMR_prj} The subspace projection and extraction for BPLMR also follow PLMR closely. In particular, let $U_k \in \mathbb{C}^{n \times (m+2)q}$ contain orthonormal basis vectors of (\ref{spaceBPLMR}). First, we
use the standard Rayleigh-Ritz procedure to obtain the projected Hermitian eigenproblem $U_k^*T(\nu)U_ky=0$, whose  
$(m+2)q$ eigenvalues (the Ritz values) 
satisfy
the nonlinear variational principle (Theorem \ref{nlvarprinciple}) and the nonlinear Cauchy interlacing theorem (Theorem \ref{cauchyilt}).
Next, we find the $r$ Ritz values 
$\nu_{1},\ldots,\nu_{r}$ that are closest to $\sigma$, and order them according to the eigenresidual norms 
of the corresponding Ritz pairs, so that~\eqref{eq:sort} holds for any $1\leq i < j\leq r$. 
Then, given the $r$ candidates, we choose $s$ Ritz values~$\nu_i$, and the associated Ritz vectors $z_i = U_k y_i$,  
that yield the smallest eigenresiduals. 
Finally, out of these $s$ Ritz pairs, 
we select the $q$ Ritz values 
$\nu_{\ell_1},\ldots, \nu_{\ell_q}$ that are closest to $\sigma$,
and further use them in the refinement procedure.

As we have already explained, the motivation for this approach is to filter out spurious Ritz values by
first 
including
all promising approximations 
in a relatively large set of $r$ Ritz pairs,
and then abandoning those with largest eigenresidual norm to obtain a 
set of $s$ candidates. This set is used to choose the $q$ most promising Ritz values, i.e., those closest to $\sigma$. 
By default, we let $r=\min\left((m+2)q,\max(3q,\ceil{(m+2)q/3})\right)$ and $s = {2q}$. 

Finally, given the Ritz values $\nu_{\ell_1},\ldots, \nu_{\ell_q}$, we use them in the refinement step
\begin{equation}\label{eq:ref_block}
y^{(i)}_{MR} = \mathrm{argmin}_{\|y\|=1}\left\|T(\nu_{\ell_i}) U_k y\right\|_2, \quad 1 \leq i \leq q.
\end{equation}
Solving $q$ minimization problems~\eqref{eq:ref_block} allows defining  
the block $X_{k+1}=[x_{k+1}^{(1)}\,\ldots\,x_{k+1}^{(q)}]$ 
of new approximate eigenvectors, such that $x^{(i)}_{k+1}=U_ky^{(i)}_{MR}$.
The corresponding values of the Rayleigh fuctional are then placed
on the diagonal of $\Phi_{k+1}$.

\subsection{Refined projection for semi-simple and tightly clustered eigenvalues}\label{sec_RFJ_repeatedev}

The refined projection described in Section \ref{sec_BPLMR_prj} can generate improved 
approximations to individual eigevectors if all the targeted eigenvalues 
$\lambda_{1},\ldots,\lambda_{q}$ 
are simple and well separated. However, if an eigenvalue of interest is semi-simple, i.e., 
$\mathrm{dim} \left(\mathrm{null}\,T(\lambda_{i})\right) = g > 1$ 
for some $i$, or if $g$ eigenvalues are tightly clustered, then the suggested 
refinement scheme has difficulties computing the entire invariant subspace. 
In this situation, the Rayleigh-Ritz procedure generates several (up to $g$) Ritz values that are very close to each other. The refinement step~\eqref{eq:ref_block} with these Ritz values then delivers almost identical new eigenvector approximations, which leads to an inaccurate approximation to the complete eigenspace.

In order to adapt the refined projection to the case of semi-simple or tightly clustered
eigenvalues, we propose the following strategy. 
We first select the most promising Ritz values $\nu_{\ell_1}, \ldots, \nu_{\ell_q}$ and distribute them among $K$ groups $G_1, \ldots, G_K$ in such a way that all the values in one group are 
very close to each other, whereas those belonging to different groups are relatively 
well-separated. Next, we presume that the Ritz values inside each group $G_{\tau}$ 
that contains multiple elements converge to a numerically semi-simple eigenvalue,
so that each $G_{\tau}$ aims at revealing a distinct semi-simple eigenvalue. 
In this case, instead of computing individual refined eigenvector approximations for the Ritz values in $G_{\tau}$ using~\eqref{eq:ref_block},
we extract an orthonormal basis that approximates 
the entire eigenspace associated with the targeted semi-simple eigenvalue. 
This is achieved by utilizing singular vectors corresponding to several smallest singular values of the reduced matrices $T(\nu)U_k$, where $\nu$ is a representative
value for the Ritz values in the given group.

More precisely, we take $G_1$ as an example, and assume without loss of generality that it contains $g > 1$ tightly clustered Ritz values, i.e., $G_1 = \{ \nu_{\ell_{1}}, \ldots, \nu_{\ell_{g}} \}$ 
for some $1 \leq g \leq q$. 
We then find the right singular vectors $y_1,\ldots,y_g$ corresponding to the $g$ smallest singular values of the matrix $T(\nu_{\ell_{1}})U_k$,
and define the new iterates as $x_{k+1}^{(i)} = U_k y_i$, 
where $1\leq i \leq g$.
The constructed vectors $x_{k+1}^{(i)}$ deliver an orthonormal basis that is 
expected to approximate the eigenspace of a semi-simple eigenvalue 
$\lambda \approx \nu_{\ell_{1}} \approx \ldots \approx \nu_{\ell_g}$. 
Note that $y_1, \ldots, y_g$ can also be the singular vectors of any of the 
matrices $T(\nu_{\ell_{i}}) U_k$, since the values $\nu_{\ell_{i}}$
are very close to each other by construction ($1 \leq i \leq g$). 

In order to assign the Ritz values to the groups $G_1, \ldots, G_K$, 
an appropriate threshold needs be chosen to determine if several values are sufficiently close to be included into one
group. An excessively small threshold could mistakenly treat a semi-simple eigenvalue as several well-separated simple eigenvalues and thus encounter the difficulty described above (fail to generate the complete eigenspace accurately), whereas an overly large threshold could incorrectly treat several distinct simple eigenvalues as a semi-simple one,
resulting in inaccurate eigenvector approximations. 
For example, in our BPLMR implementation, the Ritz values $\nu_{\ell_{1}},  \ldots, \nu_{\ell_{g}}$ are assigned to the same group if 
\begin{eqnarray}\label{eq:group}
\max_{i = 1, \ldots, g} \frac{|\nu_{\ell_{i}}-\bar{\nu}|}{|\bar{\nu}|} \leq 10^{-8},\quad \mbox{ where } 
\bar{\nu} =\frac{\sum_{i=1}^g \nu_{\ell_{i}}}{g}\cdot
\end{eqnarray}
Whenever available, an a priori information on distribution of the desired eigenvalues can be exploited for a more flexible threshold estimation. 

It is clear that, in practice, the Ritz values of group $G_{\tau}$ can converge
to multiple tightly clustered eigenvalues, which contradicts our assumption on the
convergence to a single semi-simple eigenvalue. Nevertheless, the assumption turns
out to be non-restrictive.  
{In fact, a group of tightly clustered eigenvalues can be considered as those arising from a small perturbation imposed on a semi-simple eigenvalue. Consequently, the invariant subspace associated with this group comes from a slight perturbation of the eigenspace corresponding to this semi-simple eigenvalue. In this case, it is not necessary, and in fact impractical, to compute each individual eigenvector to very high accuracy. The orthonormal basis obtained from our proposed approach forms a good approximation to the eigenspace corresponding to the presumably semi-simple eigenvalue, and thus it provides a good approximation to the invariant subspace for the clustered eigenvalues. If there is need to resolve each individual eigenpair in this clustered group to higher accuracy, we can set the tolerance described in \eqref{eq:group} moderately smaller. However, our experience indicates that an excessively small tolerance tends to delay the convergence, if the desired eigenvalue is indeed semi-simple. }

We note that additional care needs to be taken in the refinement step to avoid repeated convergence. This is because such a refinement procedure is constructed independently for each numerically distinct Ritz value, and thus the singular vectors coming from two different residual minimization problems~\eqref{eq:ref_block} tend to be numerically linearly dependent whenever two selected Ritz values $\nu_{\ell_i}$ and $\nu_{\ell_j}$ are close but not sufficiently close to be distributed into one group. To tackle this difficulty, for each candidate new eigenvector approximation $x_{k+1}^{(i)}=U_ky_{MR}^{(i)}$ ($1 \leq i \leq q$), we check if $\angle(x_{k+1}^{(i)}, \mathcal{X})$ is greater than some threshold, where $\mathcal{X}$ stands for the space spanned by all previously selected new eigenvector approximations $x_{k+1}^{(1)},\ldots,x_{k+1}^{(i-1)}$. We accept such a candidate if this criterion is satisfied; otherwise, we choose the singular vector associated with the next smallest singular value and test this condition again, until a linearly independent new eigenvector approximation $x_{k+1}^{(i)}$ is found.

\subsection{Deflation} Deflation plays a crucial role in simultaneous calculation of several eigenpairs. 
It allows eigensolvers to exclude the converged quantities from the computation and update only unconverged eigenvector approximations. It also ensures that no repeated convergence occurs. 
%
For linear eigenproblems, deflation is based on the eigen-decomposition (Hermitian case) or the Schur form (non-Hermitian case), and is usually fulfilled by orthogonalizing the search subspace against the converged invariant subspace. 
Such a deflation mechanism is often 
called ``hard deflation'' (or ``hard locking''), as the converged eigenvectors 
are not explicitly included into the search subspace. For nonlinear eigenproblems $T(\lambda)v=0$, deflation is performed by working with invariant pairs directly {using} special variants of Newton-like methods \cite{Betcke.Kressner.2011}, \cite{Effenberg.2013}, \cite{Kressner.2009}, or using the infinite Arnoldi method that allows for a Schur form on a transformed linear space \cite{Jarlebring.etal.2012}, \cite{Jarlebring.etal.2014}. 

For nonlinear Hermitian eigenproblems $T(\lambda)v=0$ that satisfy the variational principle (Theorem~\ref{nlvarprinciple}), deflation can be performed 
without explicitly preserving invariant pairs, since all eigenvectors are linearly independent. One would naturally wonder if hard deflation 
is possible, e.g., through
orthogonalization based on the scalar-valued function $[\cdot,\cdot]$ defined in (\ref{defdotprod}). Unfortunately, this approach is not viable, 
as $[\cdot,\cdot]$ is not bilinear in general, and thus the Gram-Schmidt procedure does not work. 
Instead, we simply include the converged invariant subspace
into the BPLMR search subspace generated by the unconverged eigenvectors,
and, after performing the refined projection, update only the unconverged pairs.
%
This strategy is usually called ``soft deflation'' (or ``soft locking''). 

\begin{algorithm}[htbp!]
\begin{small}
\begin{center}
  \begin{minipage}{5in}
\begin{tabular}{p{0.5in}p{4.1in}}
{\bf Input}:  &  \begin{minipage}[t]{4.0in}
Initial eigenvector approximations $X_0 \in \mathbb{C}^{n \times q}$, a preconditioner $M$, the shift $\sigma \in \mathbb{R}$, and integers $r,\,s>0$;
                  \end{minipage} \\
{\bf Output}:  &  \begin{minipage}[t]{4.0in}
                 $q$ eigenpairs $(\lambda_i,v_i)$, where $\lambda_i$'s are the eigenvalues of $T(\cdot)$ closest to $\sigma$;
                  \end{minipage}
\end{tabular}
\begin{algorithmic}[1]
\STATE Set $X_0^{act} \gets X_0$, $k \gets 0$, $d \gets 0$, and compute $\Phi_0^{act} \gets \mathrm{diag}(\rho(x_0^{(1)}),\ldots,\rho(x_0^{(q)}))$. 
\STATE Set $X_0^{conv} \gets [\,]$ and $P_{-1}^{act} \gets [\,]$. 
\WHILE {convergence not reached}
  \STATE If $k > 0$, then~$P_{k-1}^{act} \gets X_k^{act}-X_{k-1}^{act}\left(X_{k-1}^{act*} X_{k-1}^{act}\right)^{-1}X_{k-1}^{act*} X_k^{act}$. 
   \STATE Compute an orthonormal basis $U_k$ of the search subspace~\eqref{eq:bplmr_subsp_defl}, 
where the action of the preconditioner $\mathbb{M}_\mathbf{\Pi}$ is given by~\eqref{pcdmvp1} with 
$Z_k \equiv Z_k^{act} = \mathbb{T}'(X_k^{act},\Phi_k^{act})$.
   \STATE Solve the Rayleigh-Ritz projected eigenproblem~\eqref{eq:rr} for all Ritz pairs. 
   \STATE 
Restore the converged $d$ Ritz pairs, then select the $r$ \emph{unconverged} Ritz values closest to $\sigma$, and sort the corresponding Ritz pairs according to their eigenresidual norms~\eqref{eq:sort}. Then choose the $s$ Ritz pairs with minimal eigenresidual, and take the $q-d$ Ritz values $\nu_{\ell_1},\ldots,\nu_{\ell_{q-d}}$ closest to $\sigma$ from the $s$ candidates.  
   \STATE Distribute the $q-d$ Ritz values among $K$ groups $G_1,\ldots,G_K$, such that \eqref{eq:group} is satisfied. That is, the values in the same group are tightly clustered, and those in different groups are well-separated.\\
   \STATE 
(a) For $\tau = 1, 2, \ldots, K$, let $G_\tau$ be the current group containing $g(\tau)$ Ritz values, and $\nu_{\ell(\tau)}$ be a Ritz value in $G_\tau$. Find the smallest singular values of $T(\nu_{\ell(\tau)})U_k$ and associated right singular vectors $y_1,\,y_2,\,\ldots$ \\
(b) For $i=1,2,\ldots,d+(m+2)(q-d)$, compute the candidate new eigenvector approximation $U_ky_i$, and accept it only if $\angle(U_ky_i, \mathcal{X}) > \delta$, where $\mathcal{X}$ is spanned by all columns of $\mathrm{X}_{k-1}^{conv}$ and all previous accepted new eigenvector approximations.\\
(c) Once $g(\tau)$ new eigenvector approximations are obtained for group $G_\tau$, reorder all new eigenvector approximations such that $\big|\rho(x_{k+1}^{(d+1)})\!-\sigma\big| \leq \ldots \leq \big|\rho(x_{k+1}^{(q)})\!-\sigma\big|$. Normalize each column such that $x_{k+1}^{(i)\,*}T'(\rho(x_{k+1}^{(i)}))x_{k+1}^{(i)}=1$ for all $d+1 \leq i \leq q$. Move to process the next group until all $K$ groups are processed.
\STATE Determine the number $d$ of converged eigenvectors. Set $X_{k+1}^{conv}\gets[x_{k+1}^{(1)},\ldots,x_{k+1}^{(d)}]$, $X_{k+1}^{act} \gets [x_{k+1}^{(d+1)},\ldots,x_{k+1}^{(q)}]$, and $\Phi_{k+1}^{act} \gets \mathrm{diag}\big(\rho(x^{(d+1)}_{k+1}),\ldots,\rho(x^{(q)}_{k+1})\big)$.    \STATE $k \gets k + 1$. If $d = q$, then declare convergence.
   \ENDWHILE
%
\STATE Set $\lambda_i \gets \rho_{k}^{(i)}$; $v_i \gets x_{k}^{(i)}$. 
Return $(\lambda_i, v_i)$ for $i = 1, \ldots, q$.
\end{algorithmic}
\end{minipage}
\end{center}
\end{small}
  \caption{The BPLMR($m$) algorithm for a Hermitian eigenproblem $T(\lambda)v = 0$}
  \label{alg:bplmr}
\end{algorithm}

Specifically, assume that the first $d$ columns of $X_k$ have converged. 
We can then distinguish between the converged and unconverged columns.
The former can be placed into  the matrix $X_k^{conv}=[x_k^{(1)}\,\ldots\,x_k^{(d)}]$, whereas the latter are used to form the ``active'' block $X_k^{act}=[x_k^{(d+1)}\,\ldots\,x_k^{(q)}]$. The deflated BPLMR subspace can then be defined as   
\begin{equation}\label{eq:bplmr_subsp_defl}
\mathcal{U}_k^{\mbox{\scriptsize{BPLMR($m$)}}}=\mathrm{range}(X_k^{conv})+\mathcal{K}_{m+1}\left(\mathbb{M}_{\mathbf{\Pi}}^{\dagger}\mathbb{T}(\cdot,\Phi_k^{act}),X_k^{act}\right)+\mathrm{range}\left(P_{k-1}^{act}\right),
\end{equation}
where $\Phi_k^{act}=\mathrm{diag}(\rho(x_k^{(d+1)}),\ldots,\rho(x_k^{(q)}))$, and the block search direction 
$P_{k-1}^{act}$ is constructed according to~\eqref{eq:P} with $X_k$ and  $X_{k-1}$ replaced by $X_k^{act}$ and $X_{k-1}^{act}$, respectively. Here, $X_{k-1}^{act}$ refers to the eigenvector approximations in iteration $k\!-\!1$ that correspond to the active set in the current iteration $k$. 
Similarly, the preconditioner $\mathbb{M}_{\mathbf{\Pi}}$ is constructed  as in (\ref{stabpcd1}), with $X_k$, 
$\Phi_k$ and $Z_k$ replaced by $X_k^{act}$, $\Phi_k^{act}$ and $\mathbb{T}'(X_k^{act},\Phi_k^{act})$, respectively. 
Following the convention, we denote an orthonormal basis of~\eqref{eq:bplmr_subsp_defl} by $U_k$,
which contains $d+(m+2)(q-d)$ columns.

Given $U_k$, we perform the Rayleigh-Ritz procedure and solve the projected eigenproblem~\eqref{eq:rr} to obtain a set
of the Ritz pairs. We then recover the $d$ Ritz pairs that have previously converged. 
This is done by checking if a Ritz pair $(\nu,z)$ has both $\min_{1\leq i\leq d}|\nu -\rho^{(i)}|$ 
and $\angle(z, \mathrm{range}(X_k^{conv}))$ sufficiently small. Next, we apply the strategy discussed in 
Section~\ref{sec_BPLMR_prj} to select $q-d$ promising Ritz values from the remaining $(m+2)(q-d)$ Ritz pairs, and then use the selected Ritz values as shifts for the refined projection. The refined projection should be performed as described in Section \ref{sec_RFJ_repeatedev} to avoid repeated convergence. 
The entire scheme of BPLMR with deflation is summarized in Algorithm 2.

\subsection{Computing many successive eigenvalues} In this section, we discuss an extension of the use of PLMR methods for the computation of many successive eigenvalues. Such a computation is crucial in a variety of important applications, for example, where a large number of the lowest eigenvalues and corresponding eigenvectors are desired. Traditional PCG methods are generally most reliable in this setting, but they rely on the min-max property of eigenvalues and thus require complete deflation of all converged eigenvectors. Consequently, both the memory and arithmetic cost gradually become prohibitive as the number of desired eigenvalues, $n_d$, grows to a few hundred or above. In addition, for nonlinear Hermitian problems, the rapid increase in arithmetic cost is even more dramatic as $n_d$ grows, because soft deflation including all converged eigenvectors is needed for the Rayleigh-Ritz projection. As a result, solving a single projected eigenproblem becomes increasingly time-consuming.

To tackle this issue, we need to perform \emph{partial deflation}, instead of complete deflation, of converged eigenvectors. The motivation for partial deflation is that PLMR methods are designed to generate approximations to eigenvalues around the shift $\sigma$, and thus deflation of the eigenvectors associated with eigenvalues far from $\sigma$ is not necessary since the algorithms would not converge to those eigenvalues anyway, provided that a good preconditioner $M \approx T(\sigma)$ is available. Consequently, 
only a partial deflation of eigenvectors corresponding to eigenvalues near  $\sigma$ is sufficient to avoid repeated convergence. 

In fact, the partial deflation strategy can be easily developed based on the soft deflation we studied. Specifically, note that we can use soft deflation to avoid repeated convergence to any previously found eigenvectors, so that additional desired eigenpairs can be computed in an incremental manner. Let $W \in \mathbb{C}^{n \times \ell}$ contain $\ell$ converged eigenvectors already obtained. To deflate these eigenvectors, BPLMR simply develops the search subspace $$\mathrm{range}(W)+\mathrm{range}(X_k^{conv})+\mathcal{K}_{m+1}\left(\mathbb{M}_{\mathbf{\Pi}}^{\dagger}\mathbb{T}(\cdot,\Phi_k^{act}),X_k^{act}\right)+\mathrm{range}\left(P_{k-1}^{act}\right),$$ and treats $W$ the same way as $X_k^{conv}$. Specifically, it performs the Rayleigh-Ritz projection and obtains the $\ell+d$ converged Ritz pairs. It then finds the $q-d$ most promising Ritz values from the unconverged Ritz pairs and uses them as the shifts for the refinement procedure. New eigenvector approximations are generated as usual from the singular vectors corresponding to the smallest singular values of relevant matrices, and each candidate $U_ky$ is accepted only if $\angle\big(U_ky,\mathcal{X}\big)$ is not very small, where $\mathcal{X}$ is the space spanned by the columns of $W$, $X_k^{conv}$, and all previously selected new eigenvector approximations in iteration $k$. 

 
With the above extension of soft deflation, we now propose the `moving-window' style partial deflation for computing successive eigenvalues of $T(\cdot)$ on an interval $(a,b) \subset \mathbb{R}$. We start BPLMR with the set of converged eigenvectors $W = \emptyset$ to compute the $q$ eigenvalues $\lambda_{1,1},\ldots,\lambda_{1,q}$ closest to $\sigma_1=a$, and set the columns of $W$ be the corresponding eigenvectors $v_{1,1},\ldots,v_{1,q}$. Then we choose a nearby shift $\sigma_2 > \sigma_1$, and use BPLMR with $W$ to find the $q$ eigenvalues $\lambda_{2,1},\ldots,\lambda_{2,q}$ near $\sigma_2$. The two sets of eigenvalues should have no intersection due to the use of deflation. Then the new set of eigenvectors $v_{2,1},\ldots,v_{2,q}$ are added to $W$, and we choose a new shift $\sigma_3 > \sigma_2$ and invoke BPLMR again. At a certain step, if the current shift $\sigma_i$ is far from $\sigma_1$, for example, we remove the first set of eigenvectors $v_{1,1},\ldots,v_{1,q}$ from $W$. We also update the preconditioner when necessary to maintain rapid convergence for eigenvalues near the new shift. The maximum window size, i.e., the largest number of columns of $W$ allowed, is determined upon a trade-off between the storage cost and the occurrence of repeated convergence.

The described partial deflation strategy is critical for keeping the total computational cost roughly proportional to the total number, $n_d$, of desired eigenvalues. 
Recently, a strategy with similar motivation, called ``local numbering of eigenvalues'', has been successfully used with a basic nonlinear Arnoldi method for computing many successive eigenvalues \cite{Betcke.Voss.2014}. As we shall see in Section~6, our proposed approach is highly reliable and efficient in this problem setting.

\section{Numerical Experiments}\label{sect_num_exp}
We illustrate the performance of the PLMR methods on a few Hermitian eigenproblems satisfying the variational characterization (\ref{minmaxfm}) or (\ref{maxminfm}). We shall see that the new algorithms exhibit rapid and robust convergence towards interior eigenvalues, provided that good preconditioners are available. Unless otherwise noted, the experiments were performed on a Macbook computer running Mac OS X 10.7.5, MATLAB R2012b, with a 2.4 GHz Intel Core 2 Duo processor and 4GB 667MHz DDR2 memory.

\begin{table}[!h]
\begin{center}
{\footnotesize
\begin{tabular}{c|crc}
problem & type & order & interval  \\ \hline
$wiresaw$ & quadratic & $1024$ & $(0,3250)$  \\
$genhyper$ & quadratic & $4096$ & $(-843,0.3943)$  \\
$sleeper$ & quadratic & $16384$ & $(-16.33,-1.61)$  \\
$string$ & rational & $10000$ & $(4.4,1.2\times 10^9)$  \\
$pdde$ & nonlinear & $39601$ & $(-20.87,4.08)$  \\
$artificial$ & nonlinear & $16129$ & $(-0.43,3.34)$  \\
$Laplace2D$ & linear & $10000$ & $(0,8)$ \\
$Laplace3D$ & linear & $125000$ & $(0,12)$ \\
\end{tabular}
}
\end{center}
\caption{Description of the test problems}\label{tab_dspt_ss}
\end{table}

We choose eight Hermitian eigenproblems for the test. The six nonlinear eigenproblems have been introduced in part I of our study \cite{Szyld.Xue.2014b}, but we describe them here again to make this paper self-contained. Table \ref{tab_dspt_ss} summarizes these problems, among which the quadratic and the rational eigenproblems are constructed from the NLEVP toolbox \cite{Betcke.etal.2013}. The first quadratic eigenproblem $wiresaw$ of 
order $1024$ comes from the vibration analysis of a wiresaw, constructed by the command \texttt{nlevp(`wiresaw1',1024)}. The eigenvalues of this gyroscopic eigenproblem are purely imaginary and thus do not satisfy the variational principle \textup{(\ref{minmaxfm})} or \textup{(\ref{maxminfm})}, but they can be mapped to real eigenvalues of a transformed Hermitian eigenproblem by substituting $\lambda$ with $i\lambda$. The transformed problem has $1024$ pairs of real eigenvalues $\{\lambda_i^{\pm}\}$, where $\lambda_i^{-}=-\lambda_i^{+}$, and $\{\lambda_i^{-}\}$ and $\{\lambda_i^{+}\}$ lie in $I_{\ell}=(-3250,0)$ and $I_r=(0,3250)$, respectively. The variational principle (\ref{minmaxfm}) holds on $I_{\ell}$ and (\ref{maxminfm}) on $I_r$, respectively, and we look for the eigenvalues on $I_r$. Next, the hyperbolic quadratic problem $genhyper$ of 
order $4096$ is constructed by the command \texttt{nlevp(`genhyper',ev,[eye(4096) eye(4096)])}, where \texttt{ev} is a vector whose entries are the reciprocals of $8192$ random numbers generated by \texttt{randn} function initialized with a zero seed. The elements of \texttt{ev} are set to be the eigenvalues of this problem, $4096$ of which are distributed on the left interval $I_{\ell}=(-843,0.3943)$, and the rest lie in the right interval $I_r=(0.3943,20061)$. The variational principle (\ref{minmaxfm})  is satisfied on $I_{r}$ and (\ref{maxminfm}) on $I_{\ell}$, respectively, and we aim at solving the eigenvalues on $I_{\ell}$. The third quadratic problem $sleeper$ of the form $T(\lambda) = A_0+\lambda A_1+\lambda^2 A_2$ of order $16384$ models the oscillations of a rail track lying on sleepers. The problem is constructed by the command \texttt{nlevp(`sleeper',128)}, and then the matrix corresponding to the constant term is changed from $A_0$ to $A_0-2I$, so that the modified problem satisfies the variational principle \textup{(\ref{maxminfm})} on $(-16.33,-1.61)$. The rational eigenproblem $string$ of the form $T(\lambda) = A-\lambda B + \frac{\lambda}{\lambda-1}C$ of order $10000$ is generated by the command \texttt{nlevp(`string',10000)}; it arises in the finite element discretization of a boundary problem describing the eigenvibration of a string attached to a spring. The variational principle \textup{(\ref{maxminfm})} holds on the interval $(4.4,1.2\times 10^9)$. 

Two truly nonlinear eigenproblems are described as follows. The first arises from the modeling of a partial delay differential equation 
(\textit{pdde}) \cite{Jarlebring.thesis.2008} $u_t(x,t)=\Delta u(x,t)+a(x)u(x,t)+b(x)u(x,t-2)$ defined on $\Omega = [0,\pi] \times [0,\pi]$ for $t \geq 0$, where $a(x)=8\sin(x_1)\sin(x_2)$ and $b(x)=100|\sin(x_1+x_2)|$, with Dirichlet boundary condition $u(x,t)=0$ for all $x \in \partial \Omega$ and $t \geq 0$ 
Assume that the solution is of the form of $u(x,t)=e^{\lambda t}v(x)$. 
Using the standard 5-point stencil finite difference approximation to the Laplacian operator on a $200 \times 200$ uniform grid, we obtain an algebraic eigenproblem $T(\lambda) = \lambda I + (M+A)+e^{-2\lambda}B$, where the matrices $M$, $A$ and $B$ of order $39601$ are the discretized form of the Laplacian operator, 
respectively. The variational principle \textup{(\ref{maxminfm})} holds on the interval $(-20.87,4.08)$. The second is an artificial problem of order $16129$ of the form $T(\lambda) = -\sin\frac{\lambda}{5} A+\sqrt{\lambda+1}B+e^{-\lambda/\sqrt{\pi}}C$, where $A=I$, $B=\mathrm{tridiag}[1;-2;1]$, and $C$ comes from the standard 5-point stencil finite difference discretization of the Laplacian, based on a $128 \times 128$ uniform grid on the unit square, without scaling by the mesh size factor $\frac{1}{h^2}=128^2$ as done for the $\textit{pdde}$
problem. The variational principle  \textup{(\ref{maxminfm})} holds on $(-0.43,3.34)$. 


The eigenproblems $Laplace2D$ and $Laplace3D$ arise from the standard 5-point and 7-point stencil finite difference discretization of the Laplacian with Dirichlet boundary conditions on the unit square and unit cube, using $100 \times 100$ and $50 \times 50 \times 50$ uniform grids, respectively. Both are linear eigenproblems with the majority of eigenvalues 
being semi-simple. Although our focus is on nonlinear problems, we include these examples to demonstrate the BPLMR's capability to resolve multiplicities. The matrices $A$ are of order $10000$ ($Laplace2D$) and $125000$ ($Laplace3D$) and are generated using the {\sc matlab} function \texttt{laplacian.m} downloaded from the MATLAB Central File Exchange, developed by A.~Knyazev.

\subsection{PLMR vs.\ PLHR} 
In this section, we demonstrate that the PLMR's symmetry-preserving extraction strategy based on the 
refined Rayleigh-Ritz is crucial for the eigensolver's robustness. In particular, we compare PLMR to its version where
the refined Rayleigh-Ritz is replaced with the harmonic projection. 
To distinguish between the two schemes, we refer to the latter as the preconditioned locally harmonic residual (PLHR) algorithm. 
Note that the same name is used for an interior linear eigenvalue solver~\cite{Vecharynski.Knyazev.2014},
which only loosely relates to the approach considered here.     
%
%
The non-Hermitian nonlinear eigenproblem $U_k^*T(\sigma)^*T(\nu)U_ky=0$, encountered by PLHR within the harmonic Rayleigh-Ritz
procedure, is solved for the harmonic 
Ritz pair associated with the harmonic Ritz value closest to $\sigma$,
using the residual inverse iteration \cite{Jarlebring.Michiels.2011}\cite{Neumaier.1985}. 

Table \ref{tab_perform_plxr} summarizes the performance of the two methods for computing the eigenvalue closest to a given shift. Let us take the problem $wiresaw$ for instance to explain the results. We choose the shift $\sigma=800$, and we use the incomplete $\mathrm{LDL^T}$ factorization of $T(\sigma)$ with drop tolerance $\tau_d=0.8$ as the preconditioner to compute the eigenvalue $\lambda=801.026019$. The algorithms are terminated once the relative eigenresidual $\|r_k\|= \frac{\|T(\mu_k)x_k\|_2}{\|T(\mu_k)\|_F\|x_k\|_2}$ of the computed eigenpair $(\mu_k,x_k)$ satisfies $\|r_k\|\leq \tau_e=10^{-10}$. Starting with the same random initial approximation $x_0$, it takes $14$ iterations for both PLMR(2) and PLHR(2) to find the desired eigenvalue. 

\begin{table}[!htbp]
\begin{center}
{\footnotesize
\begin{tabular}{clllcc}
Problem  & $\sigma$ & \qquad$\lambda$ & Parameters  & PLMR(2) & PLHR(2) \\ \hline 
$wiresaw$ & $800$ & $\:\:\,\,801.026019$ & $\tau_d = 0.8$, $\tau_e = 10^{-10}$  & 14 & 14  \\  

$genhyper$ & $-300$ & $-283.145566$ & $\tau_d = 0$,\footnotemark[3] $\tau_e = 10^{-10}$  & 6 & 8 \\  
  
$sleeper$ & $-9.1$ &$-9.09813985$ & $\tau_d = 0.005$, $\tau_e = 10^{-10}$  & 19 & 19 \\  

$string$ & $4.9  \times 10^7$ & $\:\:\,\,48974187.5$ & $\tau_d = 0.06$, $\tau_e = 10^{-12}$   & 22 &  22 \\  

$artificial$ & $0.2$ & $\:\:\,\,0.19999002$ & $\tau_d = 0.01$, $\tau_e = 10^{-10}$  & 20 &  $\infty$ \\  

$pdde$ & $0$ & $\:\:\,\,0.00149342689$ & $\tau_d = 0.001$,$\tau_e = 10^{-10}$   & 14 & $\infty$  \\  
\end{tabular}
}
\end{center}
\caption{Comparison of \textup{PLMR(2)} and \textup{PLHR(2)} in iterations
}\label{tab_perform_plxr}
\end{table}

\footnotetext[3]{There is no need to construct an incomplete $\mathrm{LDL^T}$ preconditioner for $genhyper$ because $T(\sigma)$ is diagonal.}

Table \ref{tab_perform_plxr} shows that PLMR converges at least as rapidly as PLHR for the four initial tests, 
whereas PLHR fails to converge (marked as $\infty$) for the two remaining problems. The convergence failure occurs because PLHR stagnates with an eigenvalue approximation of low accuracy (e.g., with an eigenresidual norm around $10^{-4}$) not very close to $\sigma$. 
This stagnation can be fixed, however, by choosing a new shift $\sigma$ closer to the desired eigenvalue, or by using a stronger preconditioner. 

Thus, while the convergence rate of PLMR and its PLHR variant is similar,   
the former tends to be more robust with respect to the quality of preconditioners and the choice of the shift $\sigma$.
In addition, the robustness of PLMR strongly relies on a properly defined eigenvector extraction procedure, based on the symmtery-preserving refined Rayleigh-Ritz approach.  

\subsection{Effectiveness of the refinement procedure} In this section, we illustrate the importance of using the refined projection to stabilize the convergence of PLMR. We show that PLMR converges considerably more robustly than the version without the refinement step, which only discards spurious Ritz values and uses the strategy described in Section~3.3 to choose a Ritz pair as the new eigepair approximation. 

To demonstrate the effect of the refined projection, we compare PLMR(2) with the simplified variants of PLMR(2) and PLMR(4) that do not invoke the refinement procedure.
We run the three methods with the same random initial approximations, repeating the experiment 20 times, and show in Table \ref{tab_perform_refined} the number of times each method successfully finds the desired eigenvalue in $100$ iterations and the average count of preconditioned matvecs needed for the successful runs.

\begin{table}[!htbp]
\begin{center}
{\footnotesize
\begin{tabular}{crl|rr|rr|rr}
Problem & $\sigma$& Parameters & \multicolumn{2}{|c|}{PLMR(2)} & \multicolumn{2}{|c|}{PLMR(2)  w/o} & \multicolumn{2}{|c}{PLMR(4) w/o}\\
 &  &  & \multicolumn{2}{|c|}{} & \multicolumn{2}{|c}{refinement} & \multicolumn{2}{|c}{refinement} \\ \hline 
$wiresaw$ & $1000$ & $\tau_d = 10^{-3}$, $\tau_e = 10^{-10}$ & 20/20 & 12.3 & 3/20 & 95.0 & 19/20 & 30.7 \\  
$gen\_hyper$ & $-200$ & $\tau_d = 0$, $\tau_e = 10^{-10}$ & 20/20 & 11.1 & 6/20 & 26.7 & 10/20 & 15.2 \\  
$loaded\_string$ & $10^{5}$ & $\tau_d = 10^{-3}$, $\tau_e = 5\times10^{-12}$ & 20/20 & 6.0 & 16/20 & 19.3 & 20/20 & 8.0 \\  
$sleeper$ & $-2$ & $\tau_d = 10^{-3}$, $\tau_e = 10^{-10}$ & 20/20 & 11.3 & 5/20 & 41.6 & 19/20 & 58.9 \\  
$pdde$ & $-1$ & $\tau_d = 10^{-4}$, $\tau_e = 10^{-10}$ & 18/20 & 9.2 & 9/20 & 30.7 & 19/20 & 17.8 \\ 
$artificial$ & $0.5$ & $\tau_d = 10^{-4}$, $\tau_e = 10^{-10}$ & 20/20 & 11.4 & 5/20 & 12.4 & 17/20 & 14.6 \\ 
\end{tabular}
}
\end{center}
\caption{Comparison of \textup{PLMR(2)} with two variants that do not {perform the refinement step}: number of successful runs and average counts of preconditioned matvecs}
\label{tab_perform_refined}
\end{table}

We see clearly from Table \ref{tab_perform_refined} that the refined projection is crucial for the stabilization of convergence for PLMR. For example, we look for the eigenvalue of the problem $wiresaw$ closest to $\sigma  = 1000$, using the incomplete $\mathrm{LDL^T}$ preconditioner with drop tolerance $\tau_d = 10^{-3}$. The relative tolerance for the computed eigenpair is $\tau_e = 10^{-10}$. PLMR(2) always managed to find the desired eigenvalue, and it took $12.3$ preconditioned matvecs on average to achieve convergence. Without the refinement step, by contrast, this method only succeeded $3$ times, and on average it took $95$ preconditioned matvecs to converge. In addition, PLMR(4) without the refinement step converged $19$ times, and it took an average $30.7$ preconditioned matves to find the desired eigenpair. In fact, with only one exception (for the problem $pdde$), PLMR(2) exhibits more robust and rapid convergence than the two variants without the refinement step. Note that for $pdde$, PLMR(2) was only marginally less robust than PLMR(4) without refinement, but was still considerably more efficient than the latter. In fact, with a stronger incomplete $\mathrm{LDL^T}$ preconditioner with drop tolerance $10^{-5}$, PLMR(2) managed to outperform the latter in both robustness and efficiency.

\subsection{Order of local convergence} In Section \ref{sect_ca}, we presented a local convergence analysis of PLMR($m_k$), showing that the new method could exhibit local linear, quadratic and cubic convergence if $m_k$ is sufficiently large for each step $k$. Here, we provide some numerical evidence to support the analysis. We note that in general, perfect quadratic and cubic convergence  are rarely observed in practice. In addition, it is impractical to choose the optimal $m_k$ for each $k$ to achieve the expected order of convergence. Instead, we simply let $m_{k+1}=\beta m_k$, where $\beta$ is a small fixed integer, to illustrate that PLMR($m_k$) can easily achieve superlinear and superquadratic convergence.

\begin{table}[!htbp]
\begin{center}
{\small
\begin{tabular}{l|lll}
$wiresaw$ ($\sigma=800)$ & $m_k=2$ &  $m_{k+1}=2m_k$  & $m_{k+1}=3m_k$ \\ 
$\tau_d = 0.8$, $\gamma_0 = 0.42$ & $0.96$ & $2.06$ & $2.43$ \\ \hline

$sleeper$ ($\sigma=-9.1$) & $m_k=2$ &  $m_{k+1}=2m_k$  & $m_{k+1}=3m_k$ \\ 
$\tau_d = 0.01$, $\gamma_0 = 0.15$ & $0.98$ & $1.72$ & $2.51$ \\ \hline

$string$ ($\sigma=4.9 \times 10^7$) &  $m_k=3$ &  $m_{k+1}=2m_k$  & $m_{k+1}=3m_k$ \\ 
$\tau_d = 0.06$, $\gamma_0 = 0.1$ & $0.97$ & $1.76$ & $2.31$ \\ \hline

$artificial$ ($\sigma=0.2$) &  $m_k=3$ &  $m_{k+1}=2m_k$  & $m_{k+1}=4m_k$ \\ 
$\tau_d = 0.01$, $\gamma_0 = 0.2$ & $1.01$ & $2.45$ & $3.91$ \\ \hline

$pdde$ ($\sigma=0$) & $m_k=3$ &  $m_{k+1}=2m_k$  & $m_{k+1}=3m_k$ \\ 
$\tau_d = 0.001$, $\gamma_0 = 0.25$ & $0.98$ & $1.78$ & $2.26$ \\

\end{tabular}
}
\end{center}
\caption{Estimate of the order of local convergence for \textup{PLMR($m_k$)}}\label{tab_localcvg_order}
\end{table}

Table \ref{tab_localcvg_order} gives the estimates of the order of local convergence of PLMR($m_k$) for five test problems. Let us again take the problem $wiresaw$ as an example to interpret the results. We are looking for the eigenvalue closest to $\sigma = 800$, and we use the incomplete $\mathrm{LDL^T}$ factorization of $T(\sigma)$ with drop tolerance $\tau_d = 0.8$. The initial iterate is set as $x_0 = v+\gamma_0 u$, where $v$ is the desired unit eigenvector, $\gamma_0=0.42$, and $u$ is a fixed unit vector generated by {\sc matlab}'s \texttt{randn}. We let $m_k=2$ and run twenty PLMR($m_k$) steps. 
Then we record the relative eigenresidual $\|r_k\|=\frac{\|T(\rho_k)x_k\|_2}{\|T(\rho_k)\|_F\|x_k\|_2}$ for each $k$ and
generate points $\left(\log \|r_k\|, \log \|r_{k+1}\|\right)$, for which we find the corresponding linear least squares fitting $y=ax+b$. 
The estimated order of  convergence is the slope $a=0.96$ of the linear fit. Next, we let $m_0 = 2$, $m_{k+1}=2m_k$, 
and run three PLMR($m_k$) steps. Using the same approach, we obtain an estimated convergence of order $2.06$. 
Finally, we let $m_0 = 2$, $m_{k+1}=3m_k$, and we run two PLMR($m_k$) steps to get the convergence order of $2.43$. For all problems, we run $20$, $3$ and $2$ steps, respectively, to capture linear, quadratic and cubic convergence.

Our results show clearly that PLMR($m_k$) with a small fixed $m_k$ converges linearly, and it converges superlinearly and superquadratically with $m_{k+1}=2m_k$ and $m_{k+1}=3m_k$ (or $m_{k+1}=4m_k$), respectively. In fact, higher order of local convergence is observed for the problem $artificial$. We note, however, the efficiency of PLMR($m_k$) primarily depends on the total number of preconditioned matvecs, instead of the order of local convergence. Higher order of convergence is achieved with increasingly larger value of $m_k$ as the method proceeds. Overall, our experience is that the total number of preconditioned matves needed to achieve a certain level of eigenresidual tolerance largely depends on the quality of the preconditioner, and it is relatively insensitive to the order of local convergence. 

\subsection{PLMR vs.\ JD}
We now compare PLMR($m$) with the JD methods where the right-preconditioned GMRES($m$) is used as an inner solver (further referred to as JD-GMRES($m$)). Our numerical results show that PLMR($m$) is at least as efficient as, and is often superior to, JD-GMRES($m$), especially when working with search subspaces of a modest dimension.  

\begin{table}[!htbp]
\begin{center}
{\small
\begin{tabular}{l|l|lll}
$wiresaw$ &  &  $\tau_d$ = $0.05$  & $\tau_d$ = $0.5$ & $\tau_d$ = $0.8$ \\ 
$\sigma=800$ & PLMR & 6\:\:\textbf{12}\:\:14.05s ($2$) & 6\:\:\textbf{24}\:\:18.22s ($4$) & 4\:\:\textbf{28}\:\:19.66s ($7$) \\
$\tau_{e}=10^{-12}$, $\gamma_0 = 0.25$ & JD-GMRES & 6\:\:\textbf{18}\:\:20.84s ($2$) & 4\:\:\textbf{28}\:\:22.31s ($6$)  & 4\:\:\textbf{36}\:\:23.42s ($8$)\\ \hline

$genhyper$ &  & $\tau_d$ = $-$ &   & \\ 
$\sigma=-300$ & PLMR & 6\:\:\textbf{12}\:\:0.22s ($2$) & & \\
$\tau_{e}=10^{-12}$, $\gamma_0 = 0.04$ & JD-GMRES& 5\:\:\textbf{20}\:\:0.14s ($3$) & & \\ \hline

$sleeper$ &  & $\tau_d$ = $0.002$ &  $\tau_d$ = $0.01$  & $\tau_d$ = $0.02$ \\ 
$\sigma=-9.1$ & PLMR & 4\:\:\textbf{16}\:\:1.07s ($4$) & 3\:\:\textbf{39}\:\:1.81s ($13$) & 3\:\:\textbf{66}\:\:3.98s ($22$) \\
$\tau_{e}=10^{-12}$, $\gamma_0 = 0.1$ & JD-GMRES & 4\:\:\textbf{20}\:\:0.78s ($4$) & 3\:\:\textbf{42}\:\:1.09s ($13$) & 4\:\:\textbf{72}\:\:2.26s ($17$) \\ \hline

$string$  &  & $\tau_d$ = $0.01$ &  $\tau_d$ = $0.06$  & $\tau_d$ = $0.1$ \\ 
$\sigma=4.9 \times 10^7$ & PLMR & 4\:\:\:\:\,\textbf{8}\:\:0.82s ($2$) & 12\:\:\textbf{36}\:\:2.96s ($3$) & 8\:\:\textbf{72}\:\:4.73s ($9$) \\
$\tau_{e}=10^{-14}$, $\gamma_0 = 0.1$  & JD-GMRES & 5\:\:\textbf{15}\:\:0.98s ($2$) & \,\,\,7\:\:\textbf{56}\:\:2.48s ($7$) & 5\:\:\textbf{70}\:\:2.66s ($13$) \\ \hline

$artificial$ &  & $\tau_d$ = $0.002$ &  $\tau_d$ = $0.01$  & $\tau_d$ = $0.02$ \\ 
$\sigma=0.2$ & PLMR & 7\:\:\textbf{28}\:\:4.43s ($4$)  & 6\:\:\textbf{54}\:\:6.33s ($9$) & 7\:\:\textbf{343}\:\:48.67s ($49$)\\
$\tau_{e}=10^{-12}$, $\gamma_0 = 0.1$ & JD-GMRES & 9\:\:\textbf{45}\:\:5.88s ($4$) & 7\:\:\textbf{56}\:\:5.50s ($7$) & 8\:\:\textbf{385}\:\:32.68s ($54$) \\ \hline

$pdde$ &  & $\tau_d$ = $0.0002$ &  $\tau_d$ = $0.001$  & $\tau_d$ = $0.0016$ \\ 
$\sigma=0$ & PLMR & 3\:\:\textbf{24}\:\:6.99s ($8$)  & 3\:\:\textbf{30}\:\:8.25s ($10$) & 4\:\:\textbf{44}\:\:11.81s ($11$) \\
$\tau_{e}=10^{-12}$, $\gamma_0 = 0.1$ & JD-GMRES & 5\:\:\textbf{30}\:\:7.94s ($5$) & 3\:\:\textbf{39}\:\:8.41s ($12$) & 3\:\:\textbf{54}\:\:10.69s ($17$) \\ 

\end{tabular}
}
\end{center}
\caption{Comparison of \textup{PLMR($m$)} and basic \textup{JD-GMRES($m$)} in iterations, preconditioned {\bf matvecs}, CPU time and optimal values of $m$ (in parenthesis)}\label{tab_perform_plmrjd}
\end{table}

In Section \ref{sect_ca} we studied a close connection between PLMR($m$) and the basic variant of the JD-GMRES($m$) algorithm. It is now of interest to compare the two methods, both in terms of local and global convergence. 

As in the previous section, we choose a shift $\sigma$ for each test problem, construct the corresponding incomplete $\mathrm{LDL^T}$ preconditioner with certain drop tolerance,  and run PLMR($m$) and basic JD-GMRES($m$) with the same initial iterate $x_0$ to find the eigenvalue around the given shift. With a randomly generated $x_0$, PLMR($m$) with a sufficiently large $m$ always converges to the desired eigenvalue closest to $\sigma$, whereas basic JD-GMRES($m$) always misconverges to a different eigenvalue. This is what we expected, as basic JD without subspace expansion has poor global convergence, unless a fixed shift is used in the correction equation for sufficiently many steps before a good eigenvector approximation can be obtained. 
The PLMR($m$) method, in contrast, consistently exhibits a robust global convergence. 

Next, let us compare the two algorithms with optimal values of $m$ in local convergence. To construct the initial iterate, for both methods, we let $x_0 = v+\gamma_0 u$, where $v$ is the desired unit eigenvector, $u$ is a fixed unit perturbation vector generated by \texttt{random} function, and $\gamma_0$ is a small scalar representing the error of $x_0$. 

Take the problem $wiresaw$ with shift $\sigma=800$ as an example. The preconditioner used is the incomplete $\mathrm{LDL^T}$ factorization of $T(\sigma)$ with drop tolerance $\tau_d = 0.05,\,0.5$ and $0.8$, respectively. The initial iterate is constructed as $x_0 = v + \gamma_0 u$, where $\gamma_0 = 0.25$. Table \ref{tab_perform_plmrjd} shows that with $\tau_d = 0.05$, the optimal value $m$ for both PLMR($m$) and JD-GMRES($m$) is $m=2$ (in parenthesis), since it leads to the smallest total number of preconditioned matvecs. PLMR($2$) converges in $6$ steps, taking $12$ preconditioned matvecs and $14.05$ seconds, to achieve the relative tolerance $\frac{\|T(\rho_k)x_k\|_2}{\|T(\rho_k)\|_F\|x_k\|_2} \leq \tau_e = 10^{-12}$. JD-GMRES($2$) also converges in $6$ iterations, taking $18$ preconditioned matvecs and $20.84$ seconds. The results of the algorithms for higher drop tolerances of incomplete $\mathrm{LDL^T}$ preconditioners are obtained similarly. 

We see the following patterns in the performance comparison for local convergence. 

\begin{enumerate}
\item PLMR($m$) outperforms the basic JD-GMRES($m$) in the total number of preconditioned matvecs in essentially all circumstances. This is partially due to the fact that the former and the latter take $m$ and $m+1$ preconditioned matvecs, respectively, in each iteration step. Consequently, PLMR($m$) also tends to perform better in CPU time if the preconditioned matvec is expensive. This is an advantage of PLMR over other types of preconditioned eigensolvers, such as the nonlinear Arnoldi method, which converges considerably slower than JD if the preconditioner is weak \cite{Voss.2006}, \cite{Voss.2010}. 

\item As the quality of preconditioners deteriorates, the optimal values of $m$ for PLMR($m$) and basic JD-GMRES($m$) increase, and the latter tends to take less CPU time. This is natural, because as $m$ increases, the cost of 
other computational components in PLMR($m$), such as the refined Rayleigh-Ritz projection, becomes more pronounced. Such algorithmic components are 
intrinsically more expensive than the linear solver GMRES($m$) for large $m$. 
In this case, the superiority of PLMR might be retrieved by replacing the weak preconditioned operation $M^{-1} \approx T(\sigma)^{-1}$ by a stronger one, e.g., an approximate linear solve with the coefficient matrix $T(\sigma)$ to a reasonably small tolerance, e.g., $10^{-3}$ to $10^{-6}$. 
\end{enumerate}
\smallskip

In the following set of tests, we show
that PLMR($m$) is also more efficient than the full-featured JD-GMRES($m$) method with a search subspace of variable dimension for the Rayleigh-Ritz projection (also referred to as full JD with subspace acceleration). 
Specifically, we compare PLMR(5) with JD-GMRES(5) working with a search subspace of dimension $5$, $10$ and $20$ (denoted as JD-GMRES(5)+RR(5), etc.), respectively, for computing 10 eigenvalues closest to $\sigma$. Our implementation of JD is based on that described in \cite{Betcke.Voss.2004}, where the only difference is that our correction equation is formulated as in \eqref{JDcrteqn} and \eqref{JDprj}, using an identical projector for both algorithms. We let both methods start with the same random initial approximation $x_0$, repeat the experiment 10 times, and take the average of the number of preconditioned matvecs. The parameters used to run the tests, together with results, are summarized in Table \ref{tab_perform_plmrjd2}.

\begin{table}[!htbp]
\begin{center}
{\small
\begin{tabular}{c|cccc}
 &  PLMR(5) & JD-GMRES(5) & JD-GMRES(5) & JD-GMRES(5)\\ 
 & & + RR(5) & + RR(10) &  + RR(20)\\ \hline
$wiresaw$ ($\sigma=1000$) &&&& \\
$\tau_{e}=10^{-10}$, $\tau_d$ = $10^{-3}$ & $255$ & $1375$ & $484$ & $248$\\ \hline

$genhyper$ ($\sigma=-200$) & & & &\\ 
$\tau_{e}=10^{-10}$, $\tau_d$ = $0$ & $628$ & $1581$ & $636$ & $344$\\ \hline

$sleeper$ ($\sigma=-2$) & & & &\\ 
$\tau_{e}=10^{-10}$, $\tau_d$ = $10^{-3}$ & $265$ & $2281$ & $775$ & $396$\\ \hline

$string$ ($\sigma=10^5$)& & & &\\ 
$\tau_{e}=5 \times 10^{-12}$, $\tau_d$ = $10^{-3}$ & $273$ & $812$ & $371$ & $191$\\ \hline

$pdde$ ($\sigma=-1$)& & & &\\ 
$\tau_{e}=10^{-10}$, $\tau_d$ = $10^{-4}$ & $172$ & $495$ & $165$ & $148$\\ \hline

$artificial$ ($\sigma=0.5$)& & & &\\ 
$\tau_{e}=10^{-10}$, $\tau_d$ = $10^{-4}$ & $181$ & $325$ & $174$ & $110$\\ \hline

\end{tabular}
}
\end{center}
\caption{Comparison of \textup{PLMR($5$)} and \textup{JD-GMRES($5$) + Rayleigh-Ritz}: preconditioned {\bf matvecs} counts for computing \textup{10} eigenvalues around $\sigma$}\label{tab_perform_plmrjd2}
\end{table}

We see from Table~\ref{tab_perform_plmrjd2} that PLMR(5) is considerably more efficient than JD-GMRES(5) +RR(5), and is essentially at least as efficient as JD-GMRES(5)+RR(10). We note that PLMR(5) uses a search subspace of dimension $5$, whereas the two variants of JD-GMRES(5), respectively, work with search subspaces of total dimension $5+5=10$ and $5+10=15$, for the inner GMRES and outer JD iterations. 

We also tested other small values of $m$ for the two algorithms and found similar patterns in performance. The robust convergence of PLMR is ensured by the refined projection, whereas such a robustness of JD can only be achieved by the use of a large search subspace. For example, as can be seen in Table~\ref{tab_perform_plmrjd2}, the JD methods require
three to five times more storage to become comptetive to PLMR($m$). 
Thus, PLMR($m$) is more efficient in both arithmetic and storage cost 
when a search subspace is of a modest dimension.

\subsection{PLMR vs.\ BPLMR} In this section, we perform some tests to show that BPLMR is generally more competitive than PLMR in arithmetic cost for solving a group of clustered eigenvalues. Such a conclusion has been well established for linear eigenproblems, but to the best of our knowledge, this is the first time it is done in a nonlinear setting. 

\begin{table}[!htbp]
\begin{center}
{\footnotesize
\begin{tabular}{clrrr}
$wiresaw$ ($\sigma=800$) & PLMR(2) &  172 & {\bf 344} & 324.56s \\
$\tau_d = 10^{-3}$,$\tau_e = 10^{-10}$ & BPLMR(2) &  11 & {\bf 150} & 37.50s \\  \hline

$genhyper$ ($\sigma=-300$) & PLMR(2) &  155 & {\bf 310} & 5.74s \\
$\tau_d = 0$, $\tau_e = 10^{-10}$ & BPLMR(2) & 18 & {\bf 178} & 3.32s \\  \hline
 
$sleeper$ ($\sigma=-9.1$) & PLMR(2) &  117 & {\bf 234} & 22.68s\\
$\tau_d = 10^{-3}$, $\tau_e = 10^{-10}$ & BPLMR(2) & 10 & {\bf 150} &14.85s \\  \hline

$string$ ($\sigma=4.9 \times 10^7$) & PLMR(2) &  115 & {\bf 230} & 30.68s \\
$\tau_d = 10^{-3}$, $\tau_e = 10^{-12}$ & BPLMR(2) & 9 & {\bf 138} & 13.61s \\  \hline

$artificial$ ($\sigma=0.2$) & PLMR(2) & & &$\infty$\\
$\tau_d = 10^{-4}$, $\tau_e = 10^{-10}$ & BPLMR(2) & 6 & {\bf 50} & 8.02s \\  \hline

$pdde$ ($\sigma=0$) & PLMR(2) &  36 & {\bf 72} & 34.36s \\
$\tau_d = 10^{-4}$, $\tau_e = 10^{-10}$ & BPLMR(2) & 16 & {\bf 76} & 31.64s \\ 
\end{tabular}
}
\end{center}
\caption{Comparison of \textup{PLMR(2)} and \textup{BPLMR(2)} for computing $5$ eigenvalues around the shift $\sigma$, in iterations, preconditioned {\bf matvecs} and CPU time}\label{tab_perform_plmrbplmr}
\end{table}

We compare PLMR(2) and BPLMR(2) with the same random initial approximations to compute five eigenvalues around the shift $\sigma$. The drop tolerances $\tau_d$ of incomplete $\mathrm{LDL^T}$ preconditioner and eigenresidual tolerances $\tau_e$ are also given. Recall that both algorithms use soft deflation of converged eigenpairs, and PLMR computes the eigenvalues sequentially whereas BPLMR generates the desired approximations simultaneously.

Table \ref{tab_perform_plmrbplmr} shows that BPLMR(2) is more efficient than PLMR(2) in arithmetic cost and CPU time for most problems. For the problem $wiresaw$, for instance, it takes PLMR(2) $172$ iterations, equivalently $344$ preconditioned matvecs, and $324.56$ seconds, to find the desired $5$ eigenpairs. In contrast, it takes BPLMR(2) $11$ iterations, or equivalently $150$ preconditioned matvecs, and only $37.50$ seconds to converge. The performance difference for the two methods is minimal for the problem $pdde$. For the problem $artifical$,  PLMR(2) failed to find the fourth eigenvalue around $\sigma$ in $500$ iterations, but BPLMR(2) managed to find all the five eigenvalues. Clearly, for the same $m$, the block method is preferable in arithmetic efficiency, as long as sufficient memory is available for the larger search subspace it develops. 

\subsection{Computing many successive eigenvalues} 
To verify the reliability of the new deflation techniques, we use BPLMR with the moving-window-style partial deflation described in Section 5.5, 
to compute a large number of extreme eigenvalues. We then compare the results with those obtained by PCG methods, which are most reliable in this setting.

Table \ref{tab_perform_bplmrlobpcg} summarizes the performance of LOBPCG and BPLMR for computing a few hundred or more extreme eigenvalues of the eight test problems. 
We take the problem $string$ as an example to explain the results. The $n_d= 400$ lowest (L) eigenpairs of this problem are computed to the relative tolerance $\frac{\|T(\lambda_i)v_i\|_2}{\|T(\lambda_i)\|_F\|x_i\|_2} \leq 10^{-10}$ ($1 \leq i \leq 400$). Block methods are used to find these eigenvalues sequentially, $10$ eigenvalues each group, from the lowest to the highest ones. For each group of $10$ eigenvalues, the block size is set to be $12$ (slightly greater than 10) to stabilize the convergence. Once one group of eigenvalues are found, we compute the midpoint $\sigma=\frac{\lambda_r+\lambda_{r+1}}{2}$ between the two rightmost distinct computed eigenvalues $\lambda_r$ and $\lambda_{r+1}$, and let the new preconditioner be the $\mathrm{LDL^T}$ factorization of $T(\sigma)$. Such a preconditioner is expected to accelerate convergence towards subsequent eigenvalues near $\sigma$. 

\begin{table}[!htbp]
\begin{center}
{\scriptsize
\begin{tabular}{ccrr|rr|crrrrr}
 & & & &\multicolumn{2}{c|}{LOBPCG}  &  \multicolumn{5}{c}{LOBPCG+BPLMR(2)} \\ 
problem   & \!\!$n_d$\!\! &  \!group\! & \!block\! & precond & CPU & \!window\! & precond &  CPU & \!\!missed\!\! & \!\!repeated\!\!   \\ 
  &   &  \!size & size\! & matvecs & time & size & matvecs &  time & & \\ \hline
 $wiresaw$  & \!\!$500$\,(L)\!\! & $10$ & $12$ & { 6010} & {\bf 6566} & 4 & { 7916} & {\bf 3704} & $0$ & $0$ \\ 
 $genhyper$  &\!\!$500$\,(H)\!\!  & $10$ & $12$ & { 8657} & {\bf 7623} & 4 & { 9028} & {\bf 359} & $0$ & $0$ \\ 
  $string$  &\!\!$400$\,(L)\!\!  & $10$ & $12$ & { 5095} & {\bf 49501} & 4 & { 5843} & {\bf 1434} & $0$ & $0$ \\ 
 $artificial$  & \!\!$500$\,(H)\!\!  & $10$ & $12$ & { 5337} & {\bf 77905} & 3 & { 7317} & {\bf 2300} & $1$ & $0$ \\ 
 $pdde$  &\!\!$400$\,(H)\!\!  & $5$ &$8$ & { 4973} & {\bf 61155} & 4 & { 6457} & {\bf 3284} & $1$ & $0$ \\ 
 $sleeper$  & \!\!$501$\,(L)\!\!  & $10$ & $12$ & { 5239} & {\bf 9452} & 4 & { 5804} & {\bf 1234} & $0$ & $0$ \\ 
 $Laplace2d$ & \!\!$2001$\,(L)\!\!  & $10$ & $12$ & {\bf -} & {\bf -} & 5 & { 63749} & {\bf 13269} & $1$ & $0$ \\ 
 $Laplace3d$ & \!\!$10002$\,(L)\!\!  & $16$ & $20$ & {\bf -} & {\bf -} & 6 & { 199606} & {\bf 369806\footnotemark[4]\!\!} & $3$ & $7$ \\ 
\end{tabular}
}
\end{center}
\caption{Comparison of \textup{LOBPCG} and \textup{BPLMR} for computing successive eigenvalues: preconditioned matvecs and CPU time \textup{(}in secs\textup{)}. Eigrenresidual tolerance is $10^{-12}$ for problem $\mbox{`string'}$ and $10^{-10}$ for others. }\label{tab_perform_bplmrlobpcg}
\end{table}

\footnotetext[4]{Performed on an iMac desktop computer running Mac OS X 10.8.5, MATLAB R2012b, with a 2.9 GHz Intel Core i5 processor and 16GB 1600MHz DDR3 memory.}

To illustrate the performance of our new method, we use LOBPCG to find the lowest four (window size) blocks of eigenvalues, and then run BPLMR(2) with the moving-window-style partial deflation of the most recently converged window-sized blocks of eigenvalues. This approach is compared with LOBPCG alone for computing all desired eigenvalues. As Table~\ref{tab_perform_bplmrlobpcg} shows, it takes LOBPCG $5095$ preconditioned matvecs, and $49501$ seconds to find the $400$ lowest eigenvalues, whereas it takes the new method $5843$ preconditioned matvecs, and only $1434$ seconds. Our new approach not only avoided repeated convergence by partial deflation, but also did not miss any eigenvalue for this problem. {The CPU time of the new method is lower, because it uses partial deflation, whereas the orthogonalization costs needed for complete deflation are the bottleneck in LOBPCG, despite that the latter has lower matvec counts.} Similarly for the problem $artificial$, the highest (H) $500$ eigenvalues are computed. It takes LOBPCG and the new method $5337$ and $7317$ preconditioned matvecs, and $77905$ and $2300$ seconds, respectively. Only $1$ eigenvalue is missed by BPLMR, and no repeated convergence occurs. We note that even block PCG methods could occasionally miss a few extreme eigenvalues of linear Hermitian eigenproblems \cite{Arbenz.etal.2005}.

As we can see, our new approach is essentially as reliable as PCG methods for computing extreme eigenvalues, but is significantly less expensive when a large number $n_d$ of eigenvalues are desired. The more eigenvalues are needed, the more advantage our method has over PCG methods. We emphasize that our test is simply an illustration of the reliability and efficiency of the new algorithm in the setting of computing all successive eigenvalues on a real interval. This method can be used to find many successive interior eigenvalues as well.

\begin{figure}[!htbp]
\begin{center}
\begin{picture}(325,380)
\put(-42,270)  {\includegraphics[width=3.0in]{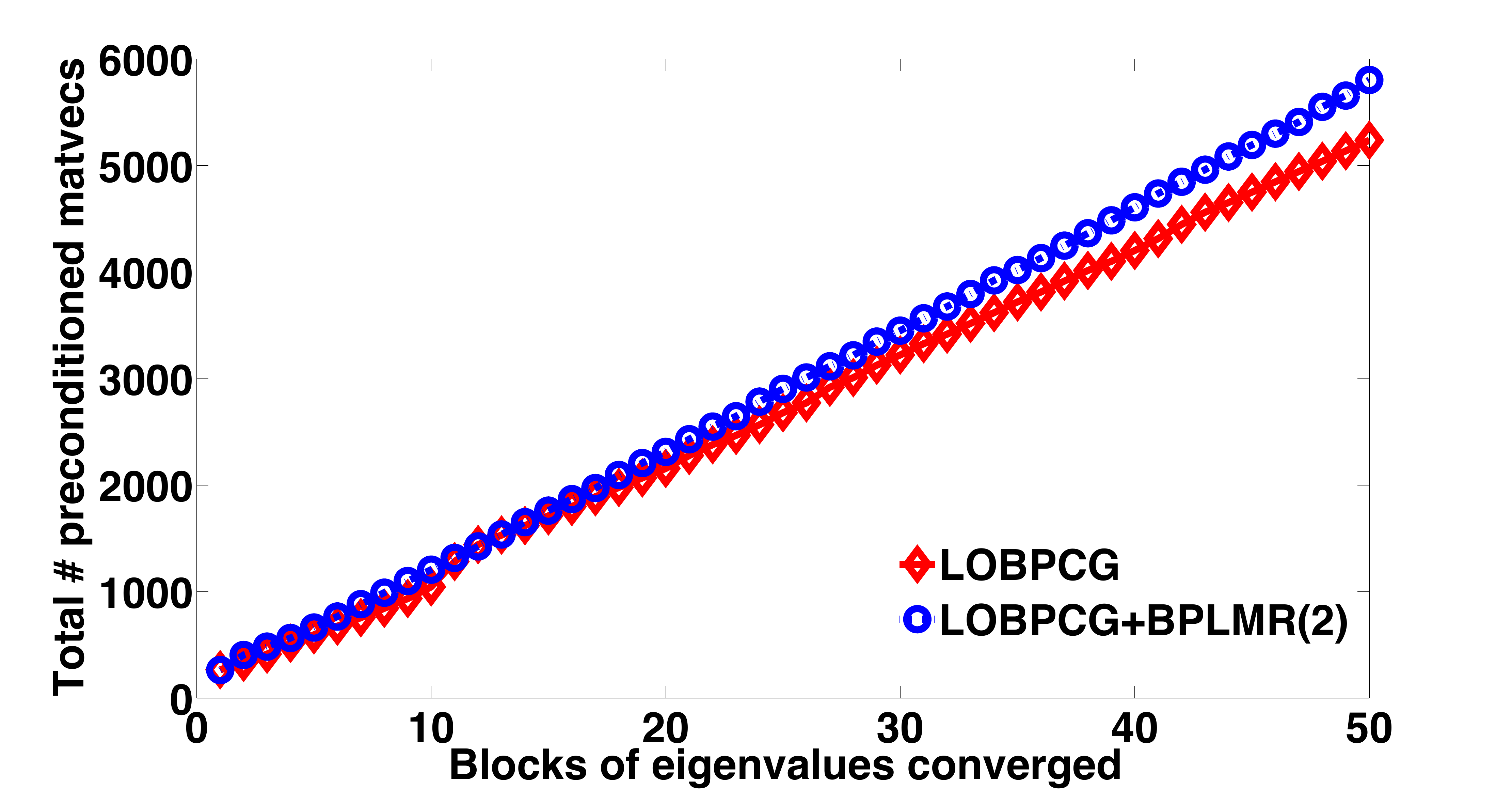}}
\put(163,270)  {\includegraphics[width=3.0in]{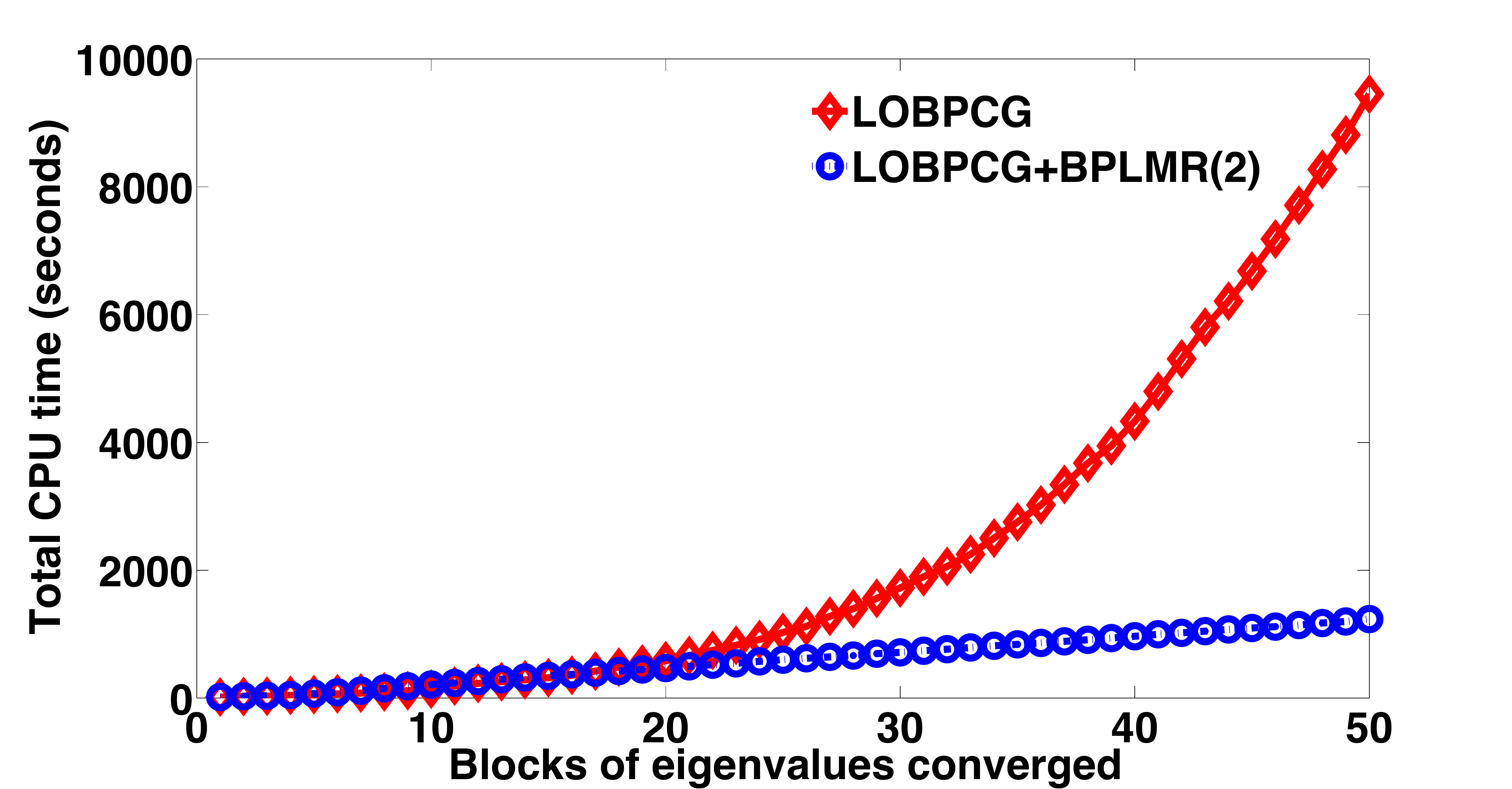}}
\put(125,255) {\includegraphics[width=1.4in]{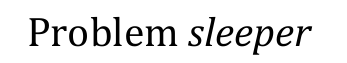}}
\put(-42,140)  {\includegraphics[width=3.0in]{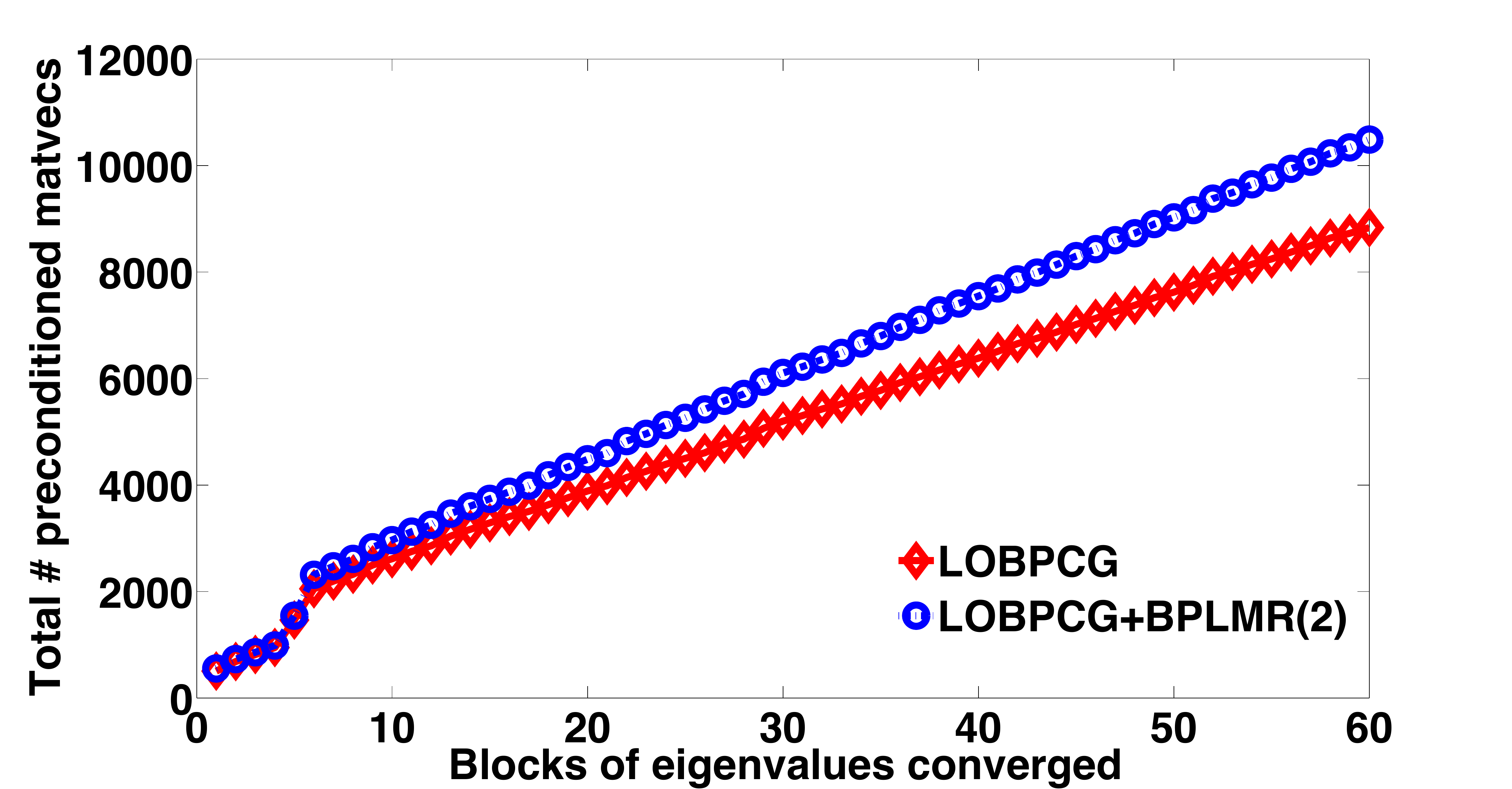}}
\put(163,140)  {\includegraphics[width=3.0in]{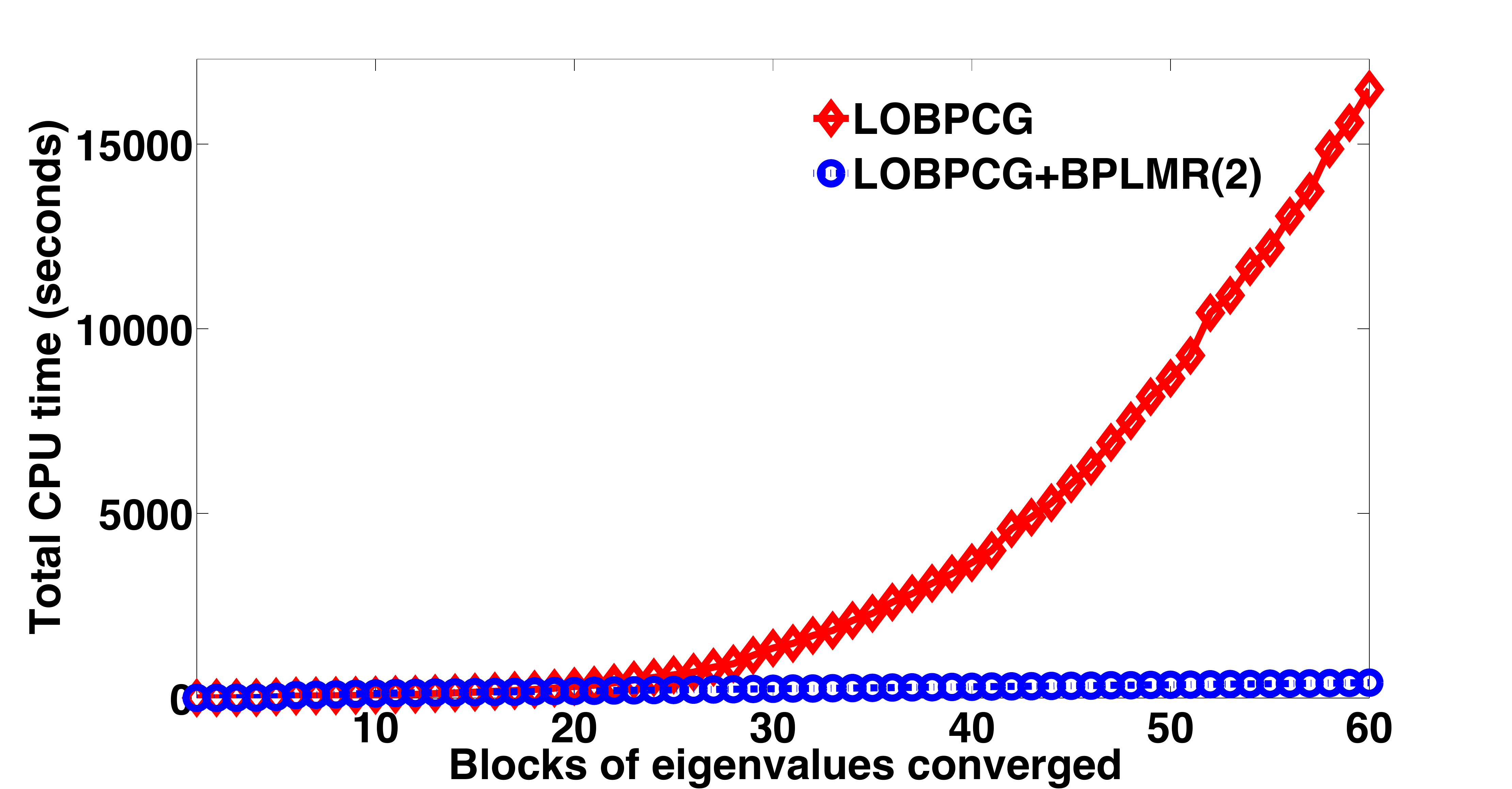}}
\put(125,125) {\includegraphics[width=1.45in]{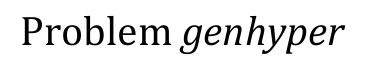}}
\put(-42,10)  {\includegraphics[width=3.0in]{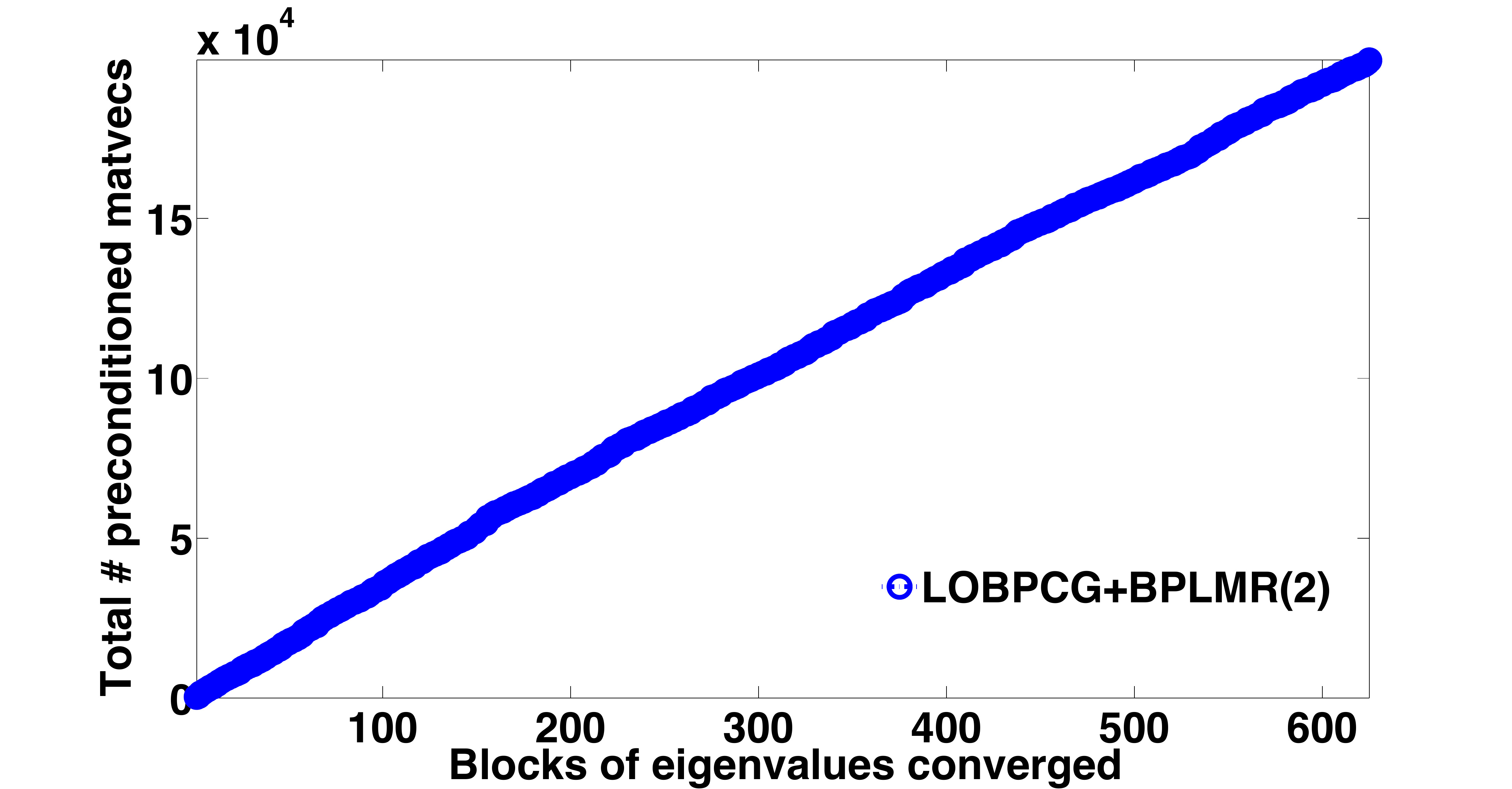}}
\put(163,10)  {\includegraphics[width=3.0in]{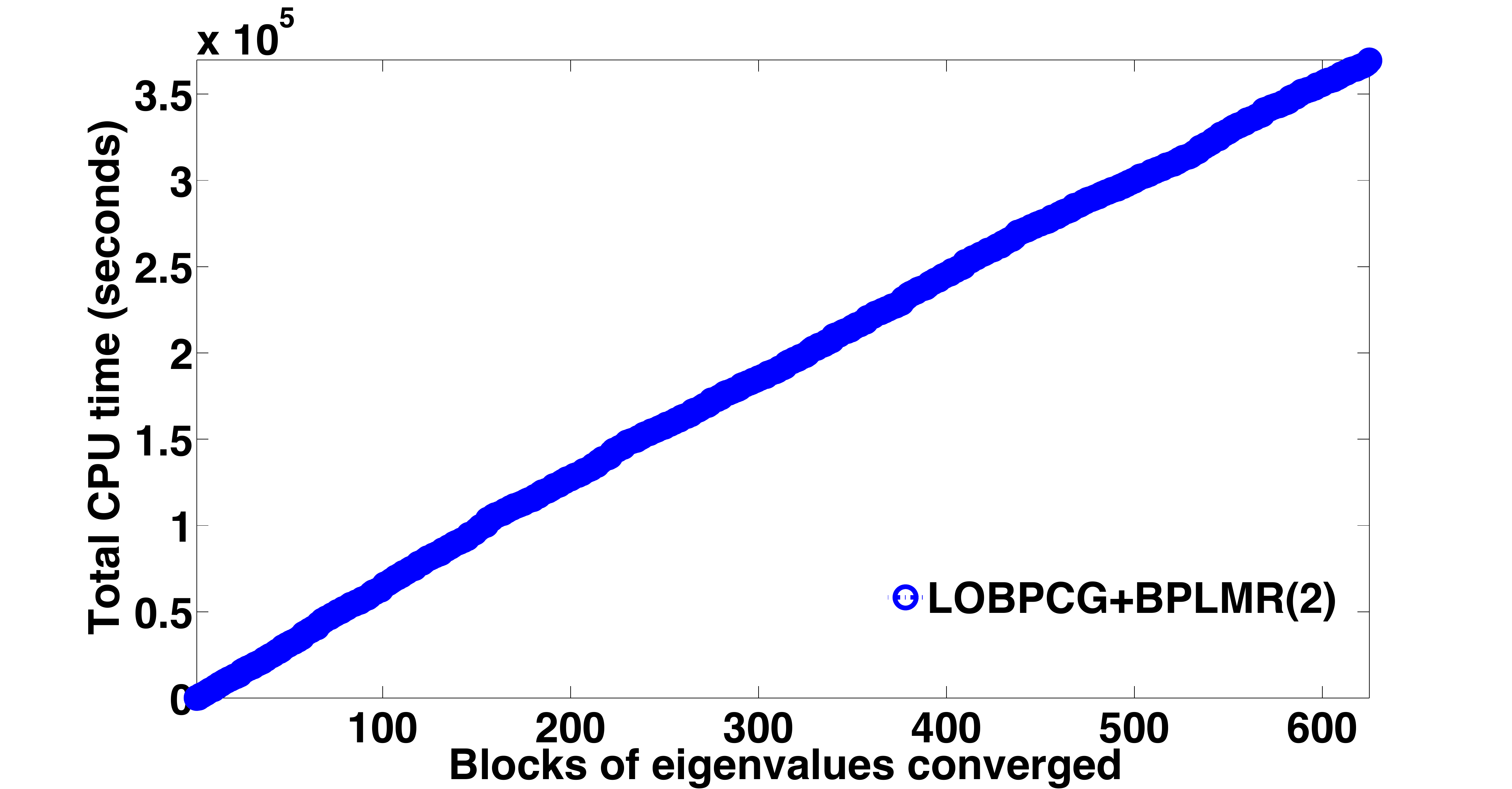}}
\put(125,-5) {\includegraphics[width=1.5in]{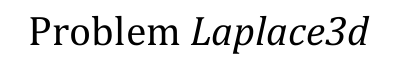}}
\end{picture}
\end{center}
\caption{Total number of preconditioned matrix-vector products and CPU time used by \textup{LOBPCG} and \textup{LOBPCG+BPLMR(2)}.\quad Top: $sleeper$, lowest $500$ eigenvalues \,\, Medium: $genhyper$, highest $600$ eigenvalues\,\,\, Bottom: $Laplace3d$, lowest $10002$ eigenvalues \textup{(}only \textup{LOBPCG+BPLMR(2)} is used\textup{)} 
} \label{fig_cvg_orders}
\end{figure}

Table \ref{tab_perform_bplmrlobpcg} also shows that BPLMR is quite reliable to find semi-simple eigenvalues with correct multiplicities. As usual, we count a distinct eigenvalue with multiplicity $\ell$ as $\ell$ eigenvalues. The last three problems in the table, namely, $sleeper$, $Laplace2d$ and $Laplace3d$ all have a dominant majority of semi-simple eigenvalues. Specifically, only the lowest and the highest eigenvalues of $sleeper$ are simple, and the rest are semi-simple with multiplicity $2$. We see that BPLMR finds all the lowest $501$ eigenvalues with correct multiplicities. For $Laplace2d$, among the lowest $2001$ eigenvalues, $35$ are simple and others are semi-simple with multiplicity $2$. BPLMR obtains all these eigenvalues with correct multiplicities, with the only exception that one semi-simple eigenvalue is found with an incorrectly lowered multiplicity $1$. $Laplace3d$ is the most challenging problem, as only $15$ eigenvalues among the lowest $10002$ ones are simple, and there are $379$ {distinct} semi-simple eigenvalues with multiplicity $3$, $1391$ with multiplicity $6$, and $42$ with multiplicity $12$. BPLMR finds a vast majority of these eigenvalues correctly. It misses $3$ eigenvalues, and it converges repeatedly to $7$ eigenvalues because those eigenvalues already moved out of the window. Using a larger window size will reduce the occurrence of repeated convergence.


In terms of efficiency, the most remarkable pattern we see from Table \ref{tab_perform_bplmrlobpcg} is as follows. LOBPCG based on the optimization of Rayleigh functional values always converges in fewer iterations than BPLMR, but the latter is significantly less expensive in arithmetic cost and thus takes much less CPU time if many (a few hundred or more) eigenvalues are desired. This observation is also clearly illustrated in Figure \ref{fig_cvg_orders} for problems $sleeper$ and $genhyper$ as an example. As we explained, this is because BPLMR uses partial deflation, instead of the highly expensive complete deflation as LOBPCG does. Moreover, BPLMR based on partial deflation only needs a \emph{fixed} amount of memory that depends on the block size, the search subspace dimension and the window size, but \emph{not} on the total number of desired eigenvalues $n_d$. The converged eigenvectors that have moved out of the window can be put on external storage because they will not be involved in subsequent computation of new eigenvalues. We see from Table \ref{tab_perform_bplmrlobpcg} and Figure \ref{fig_cvg_orders} that BPLMR with the moving-window-style deflation strategy is highly competitive if a large number of successive eigenvalues are desired.









\section{Conclusion}
We have developed a Preconditioned Locally Minimal Residual (PLMR) method for computing interior eigenvalues of nonlinear Hermitian eigenproblems $T(\lambda)v=0$ that admit a variational characterization of eigenvalues. We discussed the construction of the search subspace, stabilization of preconditioning, subspace projection and extraction, deflation, local convergence, and the extension to block variants. Our new algorithms are competitive in the rate and the robustness of convergence toward desired interior eigenvalues near a given shift. We also proposed a moving-window-style partial deflation strategy that enables BPLMR to compute a large number of successive eigenvalues. Numerical experiments show that the new approach is reliable, and is dramatically more efficient than PCG methods for computing many extreme eigenvalues.

\end{document}